\documentclass[a4paper,10pt,psamsfonts]{amsart}

\usepackage[utf8]{inputenc} 

\usepackage{amsmath}
\usepackage{amstext}
\usepackage{amsfonts}
\usepackage{amsthm}
\usepackage{amssymb}

\usepackage{hyperref}
\hypersetup{colorlinks}

\newcommand{\sect}%
{
  \setcounter{equation}{0}%
  \setcounter{figure}{0}%
  \section
}

\newcommand{\R}{\mbox{$\mathbb{R}$}}
\newcommand{\N}{\mbox{$\mathbb{N}$}}

\newcommand{\E}{\mbox{$\mathbf{E}$}}

\newcommand{\F}{\mbox{$\mathcal{F}$}}

\newcommand{\D}{\mbox{$\mathcal{D}$}}

\newcommand{\LOH}{L^p(\Omega;H)}
\newcommand{\LOR}{L^p(\Omega;\mathbb{R})}
\newcommand{\dt}{\, \mathrm{d}t}
\newcommand{\ds}{\, \mathrm{d}\sigma}
\newcommand{\dWt}{\, \mathrm{d}W(t)}
\newcommand{\dWs}{\, \mathrm{d}W(\sigma)}

\newcommand{\ee}{\mathrm{e}}

\theoremstyle{plain}

\newtheorem{definition}{Definition}[section]

\newtheorem{theorem}[definition]{Theorem}
\newtheorem{lemma}[definition]{Lemma}

\newtheorem{corollary}[definition]{Corollary}

\newtheorem{assumption}[definition]{Assumption}

\theoremstyle{definition}

\newtheorem*{remark}{Remark}
\newtheorem*{remarks}{Remarks}

\newtheorem{example}[definition]{Example}

\begin{document}

\title[Optimal Error Estimates of Galerkin Methods for SPDEs --
\today]{Optimal Error Estimates\\
of Galerkin Finite Element Methods\\
for Stochastic Partial Differential Equations\\
with Multiplicative Noise}

\author[R.~Kruse]{Raphael Kruse$^*$}


\footnotetext[1]{Department of Mathematics, Bielefeld University, P.O. Box
100131, 33501 Bielefeld, Germany \\ 
supported by CRC 701 'Spectral Analysis and Topological Structures 
in Mathematics'.}


\begin{abstract}
  We consider Galerkin finite element methods for semilinear stochastic partial
  differential equations (SPDEs) with multiplicative noise and Lipschitz
  continuous 
  nonlinearities. We analyze the strong error
  of convergence for spatially semidiscrete approximations as well as a
  spatio-temporal discretization which is based on a linear implicit
  Euler-Maruyama method. In both cases we obtain optimal error estimates.

  The proofs are based on sharp integral versions of well-known error
  estimates for the corresponding deterministic linear homogeneous equation
  together with optimal regularity results for the mild solution of the SPDE.
  The results hold for 
  different Galerkin methods such as the standard finite 
  element method or spectral Galerkin approximations.
\end{abstract}

\maketitle

\sect{Introduction}
\label{sec:intro}

This article is devoted to the analysis of numerical schemes for the
discretization of stochastic partial differential equations (SPDEs) with
multiplicative noise. For the
last $15$ years this has been an very active field of research. An extensive
list of references can be found in the review article \cite{jentzen2009c}. 

Here we apply the well-established theory of Galerkin finite element methods
from \cite{thomee2006} and, in combination with optimal spatial and temporal
regularity results \cite{jentzen2010b, kl2010a}, we derive optimal error
estimates for spatially semidiscrete as well as for spatio-temporal
approximation schemes. Our analysis is suitable to treat different Galerkin
methods such as the finite element method or spectral Galerkin methods in a
unified setting.

We begin with a probability space $(\Omega, \F, P)$ together with a normal
filtration $(\F_t)_{t \in [0,T]} \subset \F$. By $W \colon [0,T]
\times \Omega \to U$ we denote an adapted 
$Q$-Wiener process with values in a separable Hilbert space $(U, (\cdot,
\cdot)_U, \| \cdot \|_U)$. The covariance operator $Q \colon U \to U$ is
assumed to be linear, bounded, self-adjoint and positive semidefinite.  

Further, let $\left(H,( \cdot, \cdot ), \| \cdot
\| \right)$ be another separable Hilbert space and $A \colon D(A) \subset H \to
H$ a linear operator, which is 
densely defined, self-adjoint, positive definite, not
necessarily bounded but with compact inverse. Hence, there exists an increasing
sequence of real numbers $(\lambda_n)_{n \ge 1}$ and an orthonormal basis of
eigenvectors $(e_n)_{n \ge 1}$ in $H$ such that $Ae_n = \lambda_n e_n$ and 
\begin{align*}
  0 < \lambda_1 \le \lambda_2 \le \ldots \le \lambda_n (\to \infty).
\end{align*}
The domain of $A$ is characterized by 
\begin{align*}
  D(A) = \Big\{ x \in H \, : \, \sum_{n = 1}^{\infty} \lambda_n^2 (x,e_n)^2
  < \infty
  \Big\}. 
\end{align*}
Thus, $-A$ is the generator of an analytic semigroup of contractions which is
denoted by $E(t) = \ee^{-At}$. 

As our main example we have the following in mind: $H$ is the space $L^2 (\D)$,
where $\mathcal{D} \subset \R^d$ is a bounded domain with smooth boundary
$\partial \D$ or a convex domain with polygonal boundary. Then, for example,
let $-A$ be the Laplacian with homogeneous Dirichlet boundary conditions. 

Next, we introduce the stochastic process, which we want to approximate.
Let $X \colon [0,T] \times \Omega \to H$, $T > 0$, denote the
mild solution \cite[Ch. 7]{daprato1992} to the stochastic partial
differential equation 
\begin{align}
  \begin{split}
    \mathrm{d}X(t) + \left[ AX(t) + f(X(t)) \right] \dt &= g(X(t)) \dWt,
  \text{ for } 0 \le t \le T,\\
  X(0) &= X_0.
\end{split}
\label{eq1:SPDE}
\end{align}
Here $f$ and $g$ denote nonlinear operators which are Lipschitz continuous in
an appropriate sense. In Section \ref{sec:SPDE} we give a precise formulation
of our conditions on $f$, $g$ and $X_0$, which are also sufficient for
the existence and uniqueness of $X$. 

By definition the mild solution satisfies
\begin{align}
  X(t) = E(t) X_0 - \int_0^t E(t-\sigma) f(X(\sigma)) \ds + \int_0^t
  E(t-\sigma) g(X(\sigma)) \dWs 
  \label{eq1:mild}
\end{align}
for all $0 \le t \le T$. 

Our aim is to analyze numerical schemes which are used to approximate the
solution $X$. For an implementation one needs to
discretize the time interval $[0,T]$ as well as the Hilbert spaces $H$ and $U$,
since both are potentially high or infinite dimensional. 

In this paper we deal with the discretization of the time interval $[0,T]$ and
the Hilbert space $H$. A fully discrete approximation of the mild solution $X$,
which also includes the discretization of the space $U$, will be done in a
forthcoming paper.  

Our first result is concerned with a so called spatially semidiscrete
approximation of \eqref{eq1:mild}, that is, we only discretize with respect to
the Hilbert space $H$. 

By $(S_h)_{h \in (0,1]}$ we denote a family
of finite dimensional subspaces of $H$ consisting of spatially regular
functions. In our 
example with $H = L^2(\D)$, $S_h$ 
may be a standard finite element space or the linear span of finitely many
eigenfunctions of $A$ (see Examples \ref{ex:1} and \ref{ex:2}). 

Let the stochastic process $X_h\colon[0,T]\times \Omega \to S_h$
solve the stochastic evolution equation 
\begin{align}
  \begin{split}
    \mathrm{d}X_h(t) + [ A_h X_h(t) + P_h f(X_h(t)) ] \dt &= P_h g(X_h(t))
    \dWt, \text{ for } 0 \le t \le T,\\
    X_h(0) &= P_h X_0,
  \end{split}
  \label{eq4:SPDEsemi}
\end{align}
where $P_h$ denotes the (generalized) orthogonal projector onto $S_h$ and
$A_h\colon S_h \to S_h$ is a discrete version of the operator $A$ which will be
defined in \eqref{eq3:A_h}.

As for the continuous problem
\eqref{eq1:SPDE} there exists a unique mild solution $X_h$ to equation
\eqref{eq4:SPDEsemi} which satisfies
\begin{multline}
  X_h(t) = E_h(t) P_h X_0 - \int_{0}^{t} E_h(t-\sigma) P_h f(X_h(\sigma)) \ds
  \\+
  \int_{0}^{t} E_h(t-\sigma) P_h g(X_h(\sigma)) \dWs \text{ for } 0 \le t \le T.
  \label{eq4:mildsemi}
\end{multline}

This paper deals with the strong error of convergence. Here, strong convergence
is understood in the sense of convergence with respect to the norm 
\begin{align*}
  \| Z \|_{\LOH} = \big( \E \big[ \big\| Z \big\|^p \big]
  \big)^{\frac{1}{p}}, \; p \ge 2,
\end{align*}
where $\E$ is the expectation with respect to $P$. Therefore, strong
convergence indicates a good pathwise approximation.  

In many applications the aim is to approximate the expectation $\E [ 
\varphi(X(T)) ]$, where $\varphi$ is a smooth observable. This leads to the
concept of weak convergence which, in the context of SPDEs, is considered in
\cite{debussche2009, debussche2011,larsson2009a, hausenblas2010}. 
However, as it was shown by Giles \cite{giles2008a, giles2008b}, the strong
order of convergence is also essential for developing efficient multilevel
Monte Carlo methods for applications where the weak approximation is of
interest.  

Before we formulate our first main result let us explain two of the most
crucial parameters. First, we have the parameter $r \in [0,1)$ which controls
the spatial 
regularity of the mild solution $X$. On the other hand, the parameter $h \in
(0,1]$ governs the granularity of the spatial approximation. In our example
with $H = L^2(\D)$ and under the given assumptions, the mild solution $X(t)$
maps into the fractional 
Sobolev space $H_0^1(\D) \cap H^{1+r}(\D)$ and $h$ denotes the maximal
length of an edge in a partition of $\D$ into simplices. 

\begin{theorem}
  \label{th:semi}
  Under the assumptions of Section \ref{sec:SPDE} with $r \in [0,1)$, $p \in
  [2, \infty)$, and
  Assumption \ref{as:4} 
  there exists a constant $C$, independent of $h\in (0,1]$, such that
  \begin{align*}
    \| X_h(t) - X(t) \|_{\LOH} \le C h^{1 + r}, \text{ for all } t \in (0,T],
  \end{align*}
  where $X_h$ and $X$ denote the mild solutions to \eqref{eq4:SPDEsemi} and
  \eqref{eq1:SPDE}, respectively.
\end{theorem}

Therefore, in our example of a standard finite element semidiscretization, the
approximation $X_h$ converges with order $1+r$ to the mild solution $X$. Since
this rate coincides with the spatial regularity of $X$ it is called optimal
(see \cite[Ch.~1]{thomee2006}). 

We stress that, to the best of our knowledge, in all articles which deal with
the numerical approximation of semilinear SPDEs the obtained order of
convergence is of the suboptimal form $1 + r - \epsilon$ for any $\epsilon > 0$
(see \cite{yan2005} or \cite{hausenblas2003}, where also stronger
Lipschitz assumptions have been imposed on $f$, $g$) or the error estimates
contain a logarithmic term of the form $\log(t/h)$ as in \cite{kovacs2010}. 

Next, we consider a spatio-temporal discretization of the
stochastic partial differential equation \eqref{eq1:SPDE}. Let $k \in (0,1]$
denote a fixed time step which defines a time grid $t_j = jk$, $j =
0,1,\ldots,N_k$, with $N_k k \le T < (N_k + 1) k$.  

Further, by $X_h^j$ we denote the
approximation of the mild solution $X$ to \eqref{eq1:mild} at time $t_j$.  A
combination of the Galerkin methods together with a linear 
implicit Euler-Maruyama scheme results in the recursion
\begin{align}
  \label{eq5:scheme}
  \begin{split}
    X_h^j - X_h^{j-1} + k \big( A_h X_h^j + P_h f(X_h^{j-1}) \big) &= P_h
    g(X_h^{j-1}) \Delta W^j, \quad \text{for } j = 1,\ldots,N_k, \\
    X_h^0 &= P_h X_0, 
  \end{split}
\end{align}
with the Wiener increments $\Delta W^j := W(t_j) - W(t_{j-1})$ which are
$\F_{t_j}$-adapted, $U$-valued random variables. Consequently, 
$X_h^j$ is an $\F_{t_j}$-adapted random variable which takes values in $S_h$.

Our second main result is the analogue of Theorem \ref{th:semi}.

\begin{theorem}
  \label{th5:full}
  Under the assumptions of Section \ref{sec:SPDE} with $r\in [0,1)$, $p \in [2,
  \infty)$, and Assumptions \ref{as:4} and \ref{as:5} there exists a constant
  $C$, independent of $k,h \in (0,1]$, such that
  \begin{align*}
    \| X_h^j - X(t_j) \|_{\LOH} \le C (h^{1+r} + k^{\frac{1}{2}} ), \quad
    \text{for all } j = 1,\ldots,N_k,
  \end{align*}
  where $X$ denotes the mild solution \eqref{eq1:mild} to \eqref{eq1:SPDE} and
  $X_h^j$ is given by \eqref{eq5:scheme}.
\end{theorem}

As above, we obtain the optimal order of convergence with respect to the
spatio-temporal discretization. Since we only use the information of 
the driving Wiener process which is provided by
the increments $\Delta W^j$, it is a well-known fact \cite{clark1980} that
the maximum order of convergence of the implicit Euler scheme is
$\frac{1}{2}$. It is possible to overcome this barrier if one considers a
Milstein-like scheme as discussed in the recent paper
\cite{jentzen2010a}.

In our error analysis we use the results from \cite[Ch.\ 2 and 3]{thomee2006}
which are stated under the assumption that $-A$ is the Laplace 
operator with homogeneous Dirichlet boundary conditions. But all proofs and
techniques also hold in our more general framework of a self-adjoint, positive
definite operator $A$. Further, we use generic constants which may vary at
each appearance but are always independent of the discretization parameters $h$
and $k$.

The structure of this paper is as follows: In the next section we introduce some
additional notations and formulate the assumptions on $f$, $g$ and $X_0$ which
will be sufficient for the existence of the unique mild solution $X$ to
\eqref{eq1:SPDE} as well as for the proofs of Theorems \ref{th:semi} and
\ref{th5:full}. In Section \ref{sec:mainresult} we give a short review of
Galerkin finite element methods and we also introduce two additional
assumptions on the choice of the family of subspaces $(S_h)_{h \in (0,1]}$. As
already mentioned, Section \ref{sec:mainresult} also contains two concrete
examples of a spatial discretization.

In Section \ref{sec:galerkin} we present several lemmas which play a crucial
role in the proofs of our main results. All lemmas are concerned
with the spatially semidiscrete and fully discrete approximation of the
deterministic 
homogeneous equation \eqref{eq3:homeq}. We prove extensions of well-known
convergence results from \cite{thomee2006} to non-smooth initial data as well
as sharp integral versions.

While Sections \ref{sec:semi} and \ref{sec:full} are devoted to the proofs of
Theorems \ref{th:semi} and \ref{th5:full}, respectively, the final section
revisits the special case of SPDEs with additive noise. 

\sect{Preliminaries}
\label{sec:SPDE}

In order to formulate our assumptions on $f$, $g$ and $X_0$ we introduce the
notion of fractional 
powers of the linear operator $A$ in the same way as in
\cite{printems2001,kl2010a}. After fixing some additional notation we
give a precise formulation of our assumptions and cite the result on the
existence of a unique mild solution to \eqref{eq1:SPDE} from
\cite{jentzen2010b} and a corresponding regularity result from \cite{kl2010a}.

For any $r \in \R$ the operator $A^{\frac{r}{2}} : D(A^{\frac{r}{2}}) \to H$ is
given by  
\begin{align*}
  A^{\frac{r}{2}} x = \sum_{n = 1}^\infty \lambda_n^{\frac{r}{2}} x_n e_n 
\end{align*}
for all 
\begin{align*}
  x \in D(A^{\frac{r}{2}}) = \Big\{ x = \sum_{n = 1}^\infty x_n e_n \, : \,
  (x_n)_{n \ge 1} \subset \R \text{ with } \| x \|_{r}^2 := \sum_{n=1}^\infty
  \lambda_n^{r} x_n^2 < \infty \Big\}.
\end{align*}
By defining $\dot{H}^r := D( A^{\frac{r}{2}} )$ together with the norm
$\| \cdot \|_r$ for $r \in \R$, $\dot{H}^r$ becomes a
Hilbert space. Note that we have $\| x \|_{r} = \| A^{\frac{r}{2}} x \|$ for
all $r \in \R$ and $x \in \dot{H}^{r}$.

As usual \cite{daprato1992,roeckner2007} we introduce the separable Hilbert
space $U_0 := Q^{\frac{1}{2}}(U)$ with the inner product $(u_0,v_0)_{U_0} :=
(Q^{-\frac{1}{2}} u_0, Q^{-\frac{1}{2}} v_0)_U$. Here $Q^{-\frac{1}{2}}$
denotes the pseudoinverse. Further, the set $L^0_2$ denotes the space of all
Hilbert-Schmidt operators $\Phi \colon U_0 \to H$ with norm
\begin{align*}
  \| \Phi \|_{L^0_2}^2 := \sum_{m = 1}^\infty \| \Phi \psi_m \|^2,
\end{align*}
where $(\psi_m)_{m \ge 1}$ is an arbitrary orthonormal basis of $U_0$ (for
details see, for example, Proposition $2.3.4$ in \cite{roeckner2007}). 
We also introduce the subset $L_{2,r}^0 \subset L_2^0$, $r \ge 0$, which
is the subspace of all Hilbert-Schmidt operators $\Phi \colon U_0 \to
\dot{H}^r$ together with the norm $\| \Phi \|_{L^0_{2,r}} := \| A^{\frac{r}{2}}
\Phi \|_{L^0_2}$. 

Let $r \in [0,1]$, $p \in [2, \infty)$ be given. As in
\cite{jentzen2010b,kl2010a,printems2001} we impose the following 
conditions on $f$, $g$ and $X_0$.

\begin{assumption}
  \label{as:2}
  The nonlinear operator $f$ maps $H$ into $\dot{H}^{-1+r}$ and there
  exists a constant $C$ such that
  \begin{align*}
    \| f(x) - f(y) \|_{-1+r} \le C \| x - y \| \quad \text{ for all } x,y \in H.
  \end{align*}
\end{assumption}

\begin{assumption}
  \label{as:1}
  The nonlinear operator $g$ maps $H$ into $L^0_2$ and 
  there exists a constant $C$ such that
  \begin{align}
    \| g(x) - g(y) \|_{L^0_2} \le C \| x - y \| \quad \forall x,y \in H.    
    \label{eq1:G_lip}
  \end{align}
  Furthermore, we have that $g(\dot{H}^r) \subset L_{2,r}^0$ and 
  \begin{align}
    \| g(x) \|_{L_{2,r}^0} \le C \left( 1 + \| x \|_{r} \right) \quad \text{
    for all } x \in \dot{H}^r.
    \label{eq1:G_lin}
  \end{align}
\end{assumption}

\begin{assumption}
  \label{as:3}
  The initial
  value $X_0 : \Omega \to \dot{H}^{r+1}$ is an $\F_0$-measurable random
  variable with $\E \left[ \| X_0 \|^p_{r+1}\right] <\infty$.  
\end{assumption}

Under these conditions, by \cite[Theorem 1]{jentzen2010b}, there exists an
up to modification unique mild solution 
$X \colon [0,T] \times \Omega \to H$ to \eqref{eq1:SPDE} of
the form \eqref{eq1:mild}. Furthermore, by the regularity results
in \cite{kl2010a}, it holds true that for all $s \in [0,r+1]$ with $r \in
[0,1)$, we have
\begin{align}
  \label{eq1:reg}
  \sup_{t \in [0,T]} \E\left[ \| X(t) \|^p_{s} \right]  < \infty
\end{align}
and there exists a constant $C$ such that
\begin{align}
  \label{eq1:hoelder}
  \big(\E \left[ \| X(t_1) -
  X(t_2) \|_{s}^p\right] \big)^{\frac{1}{p}} \le C |t_1
  -t_2|^{\min(\frac{1}{2},\frac{r+1 - s}{2})}
\end{align}
for all $t_1,t_2 \in [0,T]$. 

\begin{remarks}
  1.) Of course, Assumption \ref{as:2} is satisfied under the usual condition
  that the mapping $f$ is Lipschitz continuous from $H$ to $H$. 
  The reason for our slightly weaker assumption is that under this condition
  the order of the spatial discretization error will numerically behave in the
  same way for both integrals, the Lebesgue integral which contains $f$ and the
  stochastic integral which contains $g$. 

  In addition, Assumption \ref{as:2} applies to partial
  differential equations where a fractional power of the operator $A$ is
  situated in front of a Lipschitz continuous mapping $\tilde{f} \colon H \to
  H$ as, for example, in the Cahn-Hilliard equation. 

  2.) Assumption \ref{as:3} can be relaxed to $X_0 \colon \Omega \to H$
      being an 
      $\F_0$-measurable random variable with $\E\left[ \| X_0 \|^p \right] <
      \infty$. But, 
      as with deterministic PDEs, this will lead to a 
      singularity at $t=0$ in the error estimates. 
\end{remarks}

We complete this section by collecting useful facts on the semigroup $E(t)$.
The smoothing property \eqref{eq1:smoothE} and Lemma \ref{lem:1} \emph{(ii)}
are classical results and proofs can be found in \cite{pazy1983}. A proof for
the remaining assertions is given in \cite{kl2010a}.

\begin{lemma}
  \label{lem:1}
  For the analytic semigroup $E(t)$ the following properties hold true:

  (i) For any $\nu \ge 0$ there exists a constant $C = C(\nu)$ such that
  \begin{align}
    \label{eq1:smoothE}
    \| A^\nu E(t) \| \le C t^{-\nu} \text{ for all } t > 0.
  \end{align}

  (ii) For any $0\le \nu \le 1$ there exists a constant $C = C(\nu)$ such
  that
  \begin{align*}
    \| A^{-\nu}(E(t) - I) \| \le C t^\nu \text{ for all } t \ge 0.
  \end{align*}

  (iii) For any $0 \le \nu \le 1$ there exists a constant $C = C(\nu)$ such
  that
  \begin{align*}
    \int_{\tau_1}^{\tau_2} \big\| A^{\frac{\nu}{2}} E(\tau_2 - \sigma) x
    \big\|^2 \ds \le C (\tau_2 - \tau_1)^{1-\nu} \left\| x 
    \right\|^2 \text{ for all $x \in H$,
    $0 \le \tau_1 < \tau_2$}.
  \end{align*}

  (iv) For any $0 \le \nu \le 1$ there exists a constant $C = C(\nu)$ such
  that
  \begin{align*}
    \Big\|A^{\nu} \int_{\tau_1}^{\tau_2} E(\tau_2 - \sigma) x \ds
    \Big\| \le C (\tau_2 - \tau_1)^{1-\nu} \| x \| \text{ for all $x
    \in H$, $0 \le \tau_1 < \tau_2$}.
  \end{align*}
\end{lemma}

\sect{A short review of Galerkin finite element methods}
\label{sec:mainresult}
In this section we first review the the Galerkin finite element methods used
for the  
discretization of the Hilbert space $H$. Following \cite[Ch.\ 2 and
3]{thomee2006} we recall the definition of several discrete operators 
which are connected to a sequence of finite dimensional subspaces of
$\dot{H}^1$. 
We close this section with two concrete examples,
namely the standard finite element method and a spectral Galerkin method.

Let $\left( S_h \right)_{h \in (0,1]}$ be a sequence of finite dimensional
subspaces of $\dot{H}^1$ and denote by $R_h \colon \dot{H}^1 \to S_h$ the
orthogonal projector (or Ritz projector) onto $S_h$ with respect to the inner
product 
\begin{align*}
  a(x,y) := \big( A^{\frac{1}{2}} x, A^{\frac{1}{2}} y \big), \quad \text{ for
  } x,y \in \dot{H}^1.
\end{align*}
Thus, we have
\begin{align*}
  a(R_h x, y_h) = a(x,y_h) \text{ for all } x \in \dot{H}^1,\; y_h \in S_h.
\end{align*}

Throughout this paper we assume that the spaces $(S_h)_{h \in (0,1]}$, satisfy
the following approximation property. Below we present two
examples of $A$, $H$ and spaces
$(S_h)_{h\in(0,1]}$ which fulfill this assumption. 
\begin{assumption}
  \label{as:4}
  Let a sequence $(S_h)_{h \in (0,1]}$ of finite dimensional subspaces of
  $\dot{H}^1$ 
  be given such that there exists a constant $C$ with 
  \begin{align}
    \label{eq3:Rh}
    \left\| R_h x - x \right\| \le C h^s \| x \|_{s} \text{ for all } x \in
    \dot{H}^s, \; s\in\{1,2\}, \; h \in(0,1].
  \end{align}
\end{assumption}

\begin{remark}
  Following \cite[Ch.\ 5.2]{larsson2003} or \cite[Ch.\ 1]{thomee2006} consider
  the linear (elliptic) problem to find $x \in D(A) =
  \dot{H}^2$ such that $Ax = z$ holds for a given $z \in H$. The \emph{weak} or
  \emph{variational} formulation of this problem is: Find $x
  \in \dot{H}^1$ which satisfies
  \begin{align}
    \label{eq3:elliptic}
    a(x,y) = (z,y) \quad \text{for all } y \in \dot{H}^1.
  \end{align}
  For a given sequence of finite dimensional subspaces $(S_h)_{h \in (0,1]}$
  the Galerkin approximation $x_h \in S_h$ of the weak solution $x$ is given by
  the relationship
  \begin{align}
    \label{eq3:elliptich}
    a(x_h,y_h) = (z,y_h) \quad \text{for all } y_h \in S_h.
  \end{align}
  Note that by the representation theorem $x \in \dot{H}^1$ and $x_h \in S_h$
  are uniquely determined by \eqref{eq3:elliptic} and
  \eqref{eq3:elliptich}. By the definition of the Ritz projector and since
  \eqref{eq3:elliptic} holds for all $y_h \in S_h$ we get 
  \begin{align*}
    a(R_h x, y_h) = a(x,y_h) = a(x_h,y_h) \quad \text{for all } y_h \in S_h.
  \end{align*}
  This yields $R_h x = x_h$, that is, the Ritz projection $R_h x$ coincides
  with the Galerkin approximation of the solution $x$ to the elliptic problem.
  Hence, Assumption \ref{as:4} 
  is a statement about the order of convergence of the sequence
  $(x_h)_{h \in (0,1]}$ to $x$.
\end{remark}

Next, we introduce the
mapping $A_h \colon S_h \to S_h$, which is a discrete version of the operator
$A$. For $x_h \in S_h$ we define $A_h x_h$  to be the unique element in
$S_h$ which satisfies the relationship
\begin{align}
  \label{eq3:A_h}
  a(x_h,y_h) = (A_h x_h, y_h) \text{ for all } y_h \in S_h.
\end{align}
Since
\begin{align*}
 \left( A_h x_h,y_h \right) = a(x_h,y_h) = \left( x_h, A_h y_h \right) 
 \text{ for all } x_h,y_h \in S_h,
\end{align*}
as well as 
\begin{align*}
  \left( A_h x_h, x_h \right) = a( x_h,x_h ) = \| x_h \|_{1}^2 > 0
  \quad \text{for all } x_h \in S_h, \; x_h \neq 0,
\end{align*}
the operator $A_h$ is self-adjoint and positive definite on $S_h$.
Hence, $-A_h$ is the generator of an analytic semigroup of contractions on
$S_h$, which
is denoted by $E_h(t) = \ee^{-A_h t}$. Let $\rho \ge 0$. Similar to \cite[Lemma
$3.9$]{thomee2006} one shows the smoothing property 
\begin{align}
  \label{eq3:smoothEh}
  \left\| A_h^{\rho} E_h(t) y_h \right\| \le C t^{-\rho} \left\| y_h \right\|
  \quad \text{for all } t > 0,  
\end{align}
where $C = C(\rho)$ is independent of $h\in (0,1]$.
Additionally, by the definition of $A_h$, it holds true that
\begin{align}
  \| A_h^{\frac{1}{2}} y_h \|^2 = a(y_h,y_h) = \| A^{\frac{1}{2}} y_h \|^2 = \|
  y_h \|_{1}^2 \quad \text{for all } y_h \in S_h \subset \dot{H}^1.
  \label{eq3:normAh}
\end{align}

Finally, let $P_h: \dot{H}^{-1} \to S_h$ be the generalized orthogonal
projector onto $S_h$ (see also \cite{chrysafinos2002}) defined by 
\begin{align*}
  (P_h x, y_h) = \langle x , y_h \rangle \quad \text{for all } x \in
  \dot{H}^{-1},\; y_h \in S_h, 
\end{align*}
where $\langle \cdot , \cdot \rangle = a\big(A^{-1} \cdot,\cdot \big)$ denotes
the duality pairing between $\dot{H}^{-1}$ and $\dot{H}^1$. By the
representation theorem, $P_h$ is well-defined and, when restricted to $H$, 
coincides with the standard orthogonal projector with 
respect to the $H$-inner product. 

These operators are related as follows:
\begin{align}
  A_h^{-1} P_h x = R_h A^{-1} x \quad \text{for all } x \in \dot{H}^{-1}
  \label{eq3:rel}
\end{align}
since 
\begin{align*}
  a( R_h A^{-1} x , y_h ) = a( A^{-1} x , y_h) = 
  \langle x , y_h \rangle = (P_h x ,y_h) = a(A_h^{-1} P_h x, y_h)
\end{align*}
for all $x \in \dot{H}^{-1}$, $y_h \in S_h$. Furthermore, it holds that
\begin{align}
  \label{eq3:discnorm}
  \begin{split}
    \| A_h^{-\frac{1}{2}} P_h x \| &= \sup_{z_h \in S_h}
    \frac{\big|(A_h^{-\frac{1}{2}}P_h x,z_h) \big|}{\| z_h \|} =
    \sup_{z_h \in S_h} \frac{\big|( P_h x,A_h^{-\frac{1}{2}}  z_h) \big|}{\|
    z_h \|} \\
    &= \sup_{z_h' \in S_h} \frac{\big|\langle x,z_h' \rangle
    \big|}{\|A_h^{\frac{1}{2}} z_h' \|} \le \sup_{z_h' \in S_h}
    \frac{\|x \|_{-1} \| z_h' \|_{1}}{ \|A_h^{\frac{1}{2}} z_h' \| } =  \|x
    \|_{-1}, 
  \end{split}
\end{align}
for all $x\in \dot{H}^{-1}$, where the last equality is due to
\eqref{eq3:normAh}. Having established this we also prove the following
consequence of \eqref{eq3:smoothEh}
\begin{align}
  \label{eq3:smoothEh2}
  \big\| E_h(t) P_h x \big\| = \big\| A_h^{\frac{1}{2}} E_h(t)
  A_h^{-\frac{1}{2}} P_h x \big\| \le C t^{-\frac{1}{2}} \big\|
  A_h^{-\frac{1}{2}} P_h x \big\| \le C t^{-\frac{1}{2}} \big\| x \big\|_{-1}
\end{align}
for all $x \in \dot{H}^{-1}$, $t > 0$ and $h \in (0,1]$.

The following 
assumption, which is concerned with the stability of the projector $P_h$ with
respect to the norm $\| \cdot \|_{1}$, will mainly be needed for the proof of
Lemma \ref{lem:Fkh2} \emph{(ii)} below and, consequently, also for the proof of
Theorem \ref{th5:full}.

\begin{assumption}
  \label{as:5}
  Let a family $(S_h)_{h \in (0,1]}$ of finite dimensional subspaces of
  $\dot{H}^1$ be given such that there exists a constant $C$ with
  \begin{align}
    \| P_h x \|_{1} \le C \| x \|_{1} \quad \text{ for all } x \in \dot{H}^1,
    \; h \in (0, 1].
    \label{eq3:stabPh}
  \end{align}  
\end{assumption}

We conclude this section with two examples which satisfy Assumptions \ref{as:4}
and \ref{as:5}.

\begin{example}[Standard finite element method]
  \label{ex:1}
  Assume that $H = L^2(\D)$, where $\D \subset \R^d$, $d=1,2,3$, is a bounded,
  convex domain (a polygon if $d=2$ or a polyhedron if $d =3$). 
  Let the operator $A$ be given by $Au = -\nabla \cdot ( a(x) \nabla u) + c(x)
  u$ with $c(x) \ge 0$ and $a(x) \ge a_0 > 0$ for $x\in D$ 
  with Dirichlet boundary conditions. In this case it 
  is well-known (for example, \cite[Theorem $6.4$]{larsson2003} and \cite[Ch.\
  3]{thomee2006}) that $\dot{H}^1 = H_0^1(\D)$ and $\dot{H}^2 = H^2(\D) \cap
  H^1_0(\D)$, where $H^k(\D)$, $k \ge 0$, denotes the Sobolev space of order
  $k$ and $H_0^1(\D)$ consists of all functions in $H^1(\D)$ which are zero on
  the boundary. Furthermore, the norms in $H^k(\D)$ and $\dot{H}^k$
  are equivalent in $\dot{H}^k$ for $k=1,2$ (see \cite[Lemma 3.1]{thomee2006}).

  Let $(\mathcal{T}_h)_{h\in(0,1]}$ denote a regular family of partitions of
  $\D$ 
  into  simplices, where $h$ is the maximal meshsize. We define $S_h$ to be the
  space of all continuous functions $y_h \colon \bar{\D} \to \R$, which are
  piecewise linear 
  on $\mathcal{T}_h$ and zero on the boundary $\partial \D$.
  Then $S_h \subset \dot{H}^1$ 
  and Assumption \ref{as:4} holds by \cite[Lemma $1.1$]{thomee2006} or
  \cite[Theorem $5.4.8$]{brenner2008}.  

  Further, if the family $(\mathcal{T}_h)_{h\in(0,1]}$ is quasi-uniform
  then also Assumption \ref{as:5} is satisfied. But for a more detailed
  discussion of Assumption \ref{as:5} in the context of the finite element
  method we refer to \cite{carstensen2002, carstensen2004, crouzeix1987}.
\end{example}

\begin{example}[Spectral Galerkin method]
  \label{ex:2}
  In the same situation as in Example \ref{ex:1} we further assume that $\D =
  (0,1) \subset \R$ and $-A$ is the Laplace operator with homogeneous Dirichlet
  boundary conditions. In this situation the orthonormal eigenfunctions and
  eigenvalues of the Laplace operator are explicitly known to be
  \begin{align*}
    \lambda_k = k^2\pi^2 \text{ and } e_k(y) = \sqrt{2} \sin(k\pi y) \text{ for
    all } k \in \N, k \ge 1, y \in \D.
  \end{align*}
  For $N \in \N$ set $h := \lambda_{N+1}^{-\frac{1}{2}}$ and define
  \begin{align*}
    S_h := \mathrm{span} \big\{ e_k : k = 1,\ldots,N \big\}.
  \end{align*}
  Note that $S_h \subset \dot{H}^r$ for every $r \in \R$. For $x \in \dot{H}^1$
  we represent the Ritz projection $R_h x \in S_h$ in terms of the basis
  $(e_k)_{k=1}^N$. This yields 
  $R_h x = \sum_{k = 1}^N x_{k}^h e_k$,
  where the coefficients $(x^h_{k})_{k = 1}^N$ are given by 
  \begin{align*}
    x_k^h = (R_h x, e_k) = \frac{1}{\lambda_k} (R_h x, A e_k) =
    \frac{1}{\lambda_k} a( R_h x, e_k) =  \frac{1}{\lambda_k} a( x, e_k) =
    (x,e_k).   
  \end{align*}
  Hence, in this example, the Ritz projector $R_h$ is the restriction of the
  orthogonal $L^2$-projector $P_h$ to $\dot{H}^1$. Moreover, we have
  \begin{align*}
    \big\| (I-R_h) x \big\|^2 &= \big\| (I - P_h) x \big\|^2 = 
    \sum_{k=N+1}^\infty (x,e_k)^2 = \sum_{k=N+1}^\infty \lambda_k^{-\rho}
    (x,A^{\frac{\rho}{2}}e_k)^2   \\
    &\le \lambda_{N+1}^{-\rho} \sum_{k=N+1}^\infty (
    A^{\frac{\rho}{2}}x,e_k)^2 = h^{2\rho} \|x \|_{\rho}^2 \text{ for all } x
    \in \dot{H}^\rho,\; \rho =1,2,
  \end{align*}
  since $\lambda_k^{-1} \le \lambda_{N+1}^{-1} = h^{2}$ for all $k \ge
  N+1$. Therefore, Assumption \ref{as:4} is satisfied for the spectral Galerkin
  method.  

  That Assumption \ref{as:5} holds is easily seen by
  \begin{align*}
    \big\| P_h x \big\|_{1}^2 = \Big\| A^{\frac{1}{2}} \sum_{k = 1}^N (x, e_k)
    e_k \Big\|^2 = \sum_{k=1}^N \big( A^{\frac{1}{2}} x, e_k)^2 \le \big\| x
    \big\|^2_{1} \quad \text{ for all } x \in \dot{H}^{1}.
  \end{align*}
  A detailed presentation of spectral Galerkin methods is found in
  \cite{hesthaven2007}. 
\end{example}

\sect{Error estimates of Galerkin methods for deterministic equations} 
\label{sec:galerkin}
This section extends error estimates from \cite{thomee2006} for the
discretization of the 
deterministic linear homogeneous equation 
\begin{align}
  \label{eq3:homeq}
  \frac{\mathrm{d}}{\mathrm{d}t}u(t) + A u(t) = 0,&\quad t>0, \quad \text{with
  } u(0) = x, 
\end{align}
to non-smooth initial data $x \in \dot{H}^{-1}$. We will also present suitable
integral version of these estimates which are crucial for the derivation of the
optimal error estimates.

\subsection{Error estimates for a spatially semidiscrete approximation}
\label{subsec:1}

The following two lemmas provide some useful estimates of the operator
$F_h$, which is given by $F_h(t) := E_h(t)P_h - E(t)$, $t \ge 0$. Note that
$F_h(t)x$ can be seen as the error at time $t \ge 0$ between the weak
solution $u$ to \eqref{eq3:homeq} and $u_h$ which solves the spatially
semidiscrete equation 
\begin{align*}
  \frac{\mathrm{d}}{\mathrm{d}t}u_{h}(t) + A_h u_h(t) = 0, & \quad t > 0, \quad
  \text{with } 
  u_h(0) = P_h x,
\end{align*}
for $x \in \dot{H}^{-1}$.

\begin{lemma}
  \label{lem:Fh1}
  Under the Assumption \ref{as:4} 
  the following estimates hold for the error operator $F_h$. 

  (i) Let $0 \le \nu \le \mu \le 2$. Then there exists a
  constant $C$ such that
  \begin{align*}
    \| F_h(t) x \| \le C h^\mu t^{-\frac{\mu - \nu}{2}} \| x \|_\nu \text{ for
    all } x \in \dot{H}^\nu,\; t > 0, \; h \in (0,1].
  \end{align*}

  (ii) Let $0 \le \rho \le 1$. Then there exists a constant $C$ such that
  \begin{align*}
    \| F_h(t) x \| \le C t^{-\frac{\rho}{2}} \| x \|_{-\rho} \text{ for all } x
    \in \dot{H}^{-\rho},\; t > 0, \; h \in (0,1].
  \end{align*}

  (iii) Let $0 \le \rho \le 1$. Then there exists a constant $C$ such that
  \begin{align*}
    \| F_h(t) x \| \le C h^{2-\rho} t^{-1} \| x \|_{-\rho} \text{ for all } x
    \in \dot{H}^{-\rho},\; t > 0, \; h \in (0,1].
  \end{align*}
\end{lemma}

\begin{proof}
  The proof of estimate \emph{(i)} can be found in \cite[Theorem
  3.5]{thomee2006}. 

  In order to prove \emph{(ii)} we first note that the case $\rho = 0$ is
  true by \emph{(i)}. Lemma \ref{lem:1}
  \emph{(i)} yields
  \begin{align}
    \label{eq3:smoothE-1}
    \|E(t) x\| = \|A^{\frac{1}{2}} E(t) A^{-\frac{1}{2}} x \| \le C
    t^{-\frac{1}{2}} \| x \|_{-1}.     
  \end{align}
  Together with \eqref{eq3:smoothEh2} this proves 
  \begin{align*} 
    \| F_h(t) x \| &\le \| E_h(t) P_h x\| + \|E(t) x\| \le C t^{-\frac{1}{2}}
    \| x \|_{-1} 
  \end{align*}
  for all $x \in \dot{H}^{-1}$. This settles the case
  $\rho = 1$. The intermediate cases follow by the interpolation technique
  which is demonstrated in the proof of \cite[Theorem 3.5]{thomee2006}.

  For \emph{(iii)} the case $\rho = 0$ is again covered by \emph{(i)}.
  Thus, it is enough to consider the case $\rho = 1$. First, by using
  \eqref{eq3:rel}, \eqref{eq3:smoothEh}, and \eqref{eq3:Rh}, we observe that
  \begin{align*}
    \big\| F_h(t) x \big\| &= \big\| A_h E_h(t) A_h^{-1} P_h x - A E(t) A^{-1}
    x \big\| \\ 
    &\le \big\| A_h E_h(t) P_h (R_h A^{-1} x - A^{-1} x) \big\| +
    \big\| (A_h E_h(t) P_h -  A E(t)) A^{-1} x \big\| \\
    &\le C t^{-1} \big\| (R_h - I) A^{-1} x \big\| + \Big\|
    \frac{\mathrm{d}F_h}{\mathrm{d}t}(t) A^{-1} x \Big\| \\
    &\le C t^{-1} h \big\| A^{-1} x \big\|_{1} + \Big\|
    \frac{\mathrm{d} F_h}{\mathrm{d}t} (t) A^{-1} x \Big\|.
  \end{align*}
  Since $\| A^{-1} x \|_1 = \| x \|_{-1}$ the first term is already in the
  desired form. The last term is estimated by a slightly
  modified version of \cite[Theorem 3.4]{thomee2006}, which gives
  \begin{align*}
    \Big\| \frac{\mathrm{d} F_h}{\mathrm{d}t}(t) A^{-1} x \Big\| \le C h t^{-1}
    \big\| A^{-1} x \big\|_{1}.    
  \end{align*}
  This proves the case $\rho = 1$ and the intermediate cases follow by
  interpolation.
\end{proof}

\begin{lemma}
  \label{lem:Fh2}
  Let $0 \le \rho \le 1$. Under Assumption \ref{as:4} the operator $F_h$
  satisfies the following estimates. 

  (i) There exists a constant $C$ such that
  \begin{align*}
    \Big\| \int_{0}^{t} F_h(\sigma) x \ds \Big\| \le C h^{2 - \rho} \| x
    \|_{-\rho} \text{ for all } x \in \dot{H}^{-\rho},\; t > 0, \; h \in (0,1].
  \end{align*}

  (ii) There exists a constant $C$ such that
  \begin{align*}
    \Big( \int_{0}^{t} \big\| F_h(\sigma) x \big\|^2 \ds \Big)^{\frac{1}{2}}
    \le C h^{1 + \rho} \| x \|_{\rho} \text{ for all } x \in \dot{H}^{\rho},
    \; t > 0, \; h \in (0,1].
  \end{align*}
\end{lemma}

\begin{proof}
  As in the proof of the previous lemma it is enough to show the estimates for
  $\rho =0$ and $\rho=1$. Then the intermediate cases follow by interpolation.

  The proof of \emph{(i)} with $\rho = 0$ is contained in
  the proof of \cite[Theorem $3.3$]{thomee2006}, where the notation
  \begin{align*}
    \tilde{e}(t) = \int_{0}^{t} F_h(\sigma) x \ds
  \end{align*}
  is used.  

  Here we present a proof of \emph{(i)} with $\rho =1$. To this end we use
  \eqref{eq3:rel} and find the estimate
  \begin{align*}
    \Big\| \int_{0}^{t} F_h(\sigma) x \ds \Big\| &= \Big\| \int_{0}^{t}
    \big(A_h E_h(\sigma) A_h^{-1} P_h - A E(\sigma) A^{-1} \big) x \ds \Big\|
    \\
    &\le \Big\| \int_{0}^{t}
    A_h E_h(\sigma) P_h \big( R_h - I \big) A^{-1} x \ds \Big\| \\
    &\quad + \Big\|
    \int_{0}^{t} \big(A_h E_h(\sigma)P_h - A E(\sigma)\big) A^{-1} x \ds \Big\|
    \\
    &= \Big\| \int_{0}^{t} \frac{\mathrm{d} E_h}{\mathrm{d}\sigma}
    (\sigma) P_h \big( R_h - I \big) 
    A^{-1} x \ds \Big\| + 
    \Big\| \int_{0}^{t} \frac{\mathrm{d} F_h}{\mathrm{d}\sigma} (\sigma) A^{-1}
    x \ds \Big\|. 
  \end{align*}
  By the fundamental theorem of calculus, $\| P_h y \| \le \|y\|$ for all $y
  \in H$, and Assumption \ref{as:4} we have for the first term
  \begin{align*}
    \Big\| \int_{0}^{t} \frac{\mathrm{d} E_h}{\mathrm{d} \sigma} (\sigma) P_h
    \big( R_h - I \big) 
    A^{-1} x \ds \Big\| &= \big\| \big( E_h(t) - I\big) P_h \big( R_h - I \big)
    A^{-1} x \big\| \\
    &\le C h \big\|A^{-1} x \big\|_{1} = Ch \| x \|_{-1}.    
  \end{align*}
  For the second term we use Lemma \ref{lem:Fh1}
  \emph{(i)} with $\mu = \nu = 1$. This yields
  \begin{align*}
    \Big\| \int_{0}^{t} \frac{\mathrm{d} F_h}{\mathrm{d} \sigma} (\sigma)
    A^{-1} x \ds \Big\| &= \big\| 
    \big( F_h(t) - F_h(0) \big) A^{-1} x \big\| \\
    &\le \big\| F_h(t) A^{-1} x \big\| + \big\| (I - P_h) A^{-1} x \big\| \le C
    h \| x \|_{-1}.
  \end{align*}
  In the last step we used the best approximation property of the
  orthogonal projector $P_h$, which, together with \eqref{eq3:Rh}, gives
  $$\big\| (P_h - I) y \big\| \le \big\| (R_h - I)
  y \big\| \le Ch\|y\|_{1} \text{ for all } y \in \dot{H}^1.$$

  It remains to prove \emph{(ii)}. From \cite[(2.28)]{thomee2006} we have the
  inequality
  \begin{align*}
    \int_{0}^{t} \big\| F_h(\sigma) x \big\|^2 \ds \le \int_0^{t} \big\| (R_h -
    I) E(\sigma) x \big\|^2 \ds.    
  \end{align*}
  In both cases, $\rho \in \{0,1\}$, we have by \eqref{eq3:Rh}
  \begin{align*}
    \big\| (R_h - I) E(\sigma) x \big\| \le  C h^{1 + \rho} \big\| E(\sigma) x
    \big\|_{1+\rho} = C h^{1 + \rho} \big\| A^{\frac{1}{2}} E(\sigma)
    A^{\frac{\rho}{2}} x \big\|.
  \end{align*}
  Applying Lemma \ref{lem:1} \emph{(iii)} with $\nu = 1$ completes the proof.
\end{proof}

\subsection{Error estimates for a fully discrete approximation}
\label{subsec:2}

In this subsection we consider a fully discrete approximation of the homogeneous
equation \eqref{eq3:homeq}. Our method of choice is a combination of the
spatial Galerkin discretization together with the well-known
implicit Euler scheme. As in Subsection \ref{subsec:1} let a family of
subspaces $(S_h)_{h \in (0,1]} \subset \dot{H}^1$ be given. The fully discrete
approximation scheme is defined by the recursion 
\begin{align}
  \label{eq3:homeq_disc}
  U^j_h + k A_h U^j_h = U^{j-1}_h,& \quad j = 1,2, \ldots \quad \text{with }
  U^0_h = P_h x. 
\end{align}
Here $k \in (0,1]$ denotes a fixed time step and $U_h^j \in S_h$ denotes the
approximation of 
$u(t_j)$ at time $t_j = jk$. A closed form representation of
\eqref{eq3:homeq_disc} is given by
\begin{align}
  \label{eq3:homeq_disc2}
  U^j_h = (I + kA_h)^{-j} P_h x, \quad j=0,1,2,\dots\,.
\end{align}
In order to make the results from \cite[Ch.\ 7]{thomee2006} accessible and to
indicate generalizations to other onestep methods onestep methods we introduce
the rational function 
\begin{align*}
  R(z) = \frac{1}{1 + z} \quad \text{for } z \in \R, z \neq 1.
\end{align*}

By $R(k A_h)$ we denote the linear operator which is defined by
\begin{align}
  \label{eq3:r}
  R(k A_h) x = \sum_{m = 1}^{N_{h}} R(k \lambda_{h,m} ) (x, \varphi_{h,m} )
  \varphi_{h,m}, 
\end{align}
where $(\lambda_{h,m})_{m = 1}^{N_h}$ are the positive eigenvalues of
$A_h\colon S_h \to S_h$ with corresponding orthonormal eigenvectors
$(\varphi_{h,m})_{m = 1}^{N_h} \subset S_h$ and $\dim(S_h) = N_h$. With this
notation \eqref{eq3:homeq_disc} is equivalently written as
\begin{align}
  \label{eq3:homeq_disc3}
  U^j_h = R(k A_h)^j P_h x, \quad j = 0,1,2,\dots\, .
\end{align}
The characteristic function $R$ of the implicit Euler scheme enjoys the
following properties with $q = 1$:
\begin{align}
  \label{eq3:IV}
  \begin{split}
    R(z) = \ee^{-z} + \mathcal{O}(z^{q+1}) &\quad \text{for } z \to 0,\\
    | R(z) | < 1 \quad \text{for all } z > 0,&\quad \text{and } \lim_{z \to
    \infty} R(z) = 0.
  \end{split}
\end{align}
In the nomenclature of \cite[Ch.\ 7]{thomee2006} the rational function $R(z)$
is an approximation of $\ee^{-z}$ with accuracy $q = 1$ and is said to be of
type IV. A onestep scheme, whose characteristic rational function possesses the
properties \eqref{eq3:IV}, is unconditionally stable and satisfies, for $\rho
\in [0,1]$, the discrete smoothing property 
\begin{align}
  \label{eq3:smoothr}
  \| A^{\rho}_h R(k A_h)^j x_h \| \le C t_j^{-\rho} \| x_h \| \quad \text{for
  all } j = 1,2,\ldots \text{ and } x_h \in S_h,
\end{align}
where the constant $C= C(\rho)$ is independent of $k,h$ and $j$. For a proof of
\eqref{eq3:smoothr} we refer to \cite[Lemma\ 7.3]{thomee2006}.

Further, as in the proof of \cite[Theorem 7.1]{thomee2006} it follows from
\eqref{eq3:IV} that there exists a constant $C$ with 
\begin{align}
  \label{eq3:r1}
  | R(z) - \ee^{-z} | \le C z^{q+1} \text{ for all } z \in [0, 1] 
\end{align}
and there exists a constant $c \in (0, 1)$ with
\begin{align}
  \label{eq3:r2}
  | R(z) | \le \ee^{-c z} \text{ for all } z \in [0, 1].
\end{align}

The rest of this subsection deals with
estimates of the error between the discrete approximation $U^j_h$ and
the solution $u(t_j)$. For the error analysis in Section \ref{sec:full} it will
be convenient to introduce an error operator $$F_{kh}(t) := E_{kh}(t) P_h -
E(t),$$
which is defined for all $t \ge 0$, where
\begin{align}
  \label{eq3:defEkh}
  E_{kh}(t) := R(kA_h)^j, \quad \text{if } t \in [t_{j-1}, t_j) \text{ for }
  j=1,2,\ldots\, .
\end{align}
The mapping $t \mapsto E_{kh}(t)$, and hence $t \mapsto F_{kh}(t)$, is right
continuous with left limits. A simple consequence
of \eqref{eq3:smoothr} and \eqref{eq3:discnorm} is the inequality
\begin{align}
  \label{eq3:smoothEkh}
  \big\| E_{kh}(t) P_h x \big\| = \big\| A_h^{\frac{1}{2}} R(k A_h)^j
  A_h^{-\frac{1}{2}} P_h x \big\| \le C t_{j}^{-\frac{1}{2}} \big\| x
  \big\|_{-1} \le C t^{-\frac{1}{2}} \big\| x
  \big\|_{-1},  
\end{align}
which holds for all $x \in \dot{H}^{-1}$, $h,k \in (0,1]$ and $t > 0$ with $t
\in [t_{j-1}, t_j)$, $j = 1,2,\ldots$.

The following lemma is the time discrete analogue of Lemma \ref{lem:Fh1}.

\begin{lemma}
  \label{lem:Fkh1}
  Under Assumption \ref{as:4}
  the following estimates hold for the
  error operator $F_{kh}$.

  (i) Let $0 \le \nu \le \mu \le 2$. Then there exists a constant $C$ such that
  \begin{align*}
    \big\| F_{kh}(t) x \big\| \le C \big( h^{\mu} + k^{\frac{\mu}{2}} \big)
    t^{-\frac{\mu - \nu}{2}} \big\| x \big\|_{\nu} \text{ for all } x \in
    \dot{H}^\nu, \; t > 0, \; h, k \in (0,1].
  \end{align*}

  (ii) Let $0 \le \rho \le 1$. Then there exists a constant $C$ such that
  \begin{align*}
    \big\| F_{kh}(t) x \big\| \le C t^{-\frac{\rho}{2}} \big\| x \big\|_{-\rho}
    \text{ for all } x \in \dot{H}^{-\rho}, \; t > 0, \; h, k \in (0,1].
  \end{align*}

  (iii) Let $0 \le \rho \le 1$. Then there exists a constant $C$ such that
  \begin{align*}
    \big\| F_{kh}(t) x \big\| \le C \big( h^{2- \rho} + k^{\frac{2 - \rho}{2}}
    \big) t^{-1} \big\| x
    \big\|_{-\rho} \text{ for all } x \in \dot{H}^{-\rho}, \; t > 0, \; h, k
    \in (0,1]. 
  \end{align*}
\end{lemma}

\begin{proof}
  (i) Let $t >0$ be such that $t_{j-1} \le t < t_j$ and $x \in \dot{H}^{\nu}$.
  Then we get 
  \begin{align*}
    \big\| F_{kh}(t) x \big\| &\le \big\| \big(R(kA_h)^j P_h - E(t_j) \big) x
    \big\| + \big\| \big( E(t_j) - E(t) \big) x \big\|.
  \end{align*}
  For the second summand we have by Lemma \ref{lem:1} \emph{(i)} and \emph{(ii)}
  \begin{align*}
    \big\| \big( E(t_j) - E(t) \big) x \big\| &= \big\| A^{-\frac{\mu}{2}}
    \big( E(t_j - t) - I \big) A^{\frac{\mu - \nu}{2}} E(t)
    A^{\frac{\nu}{2}} x \big\| \\   
    &\le C (t_j - t)^{\frac{\mu}{2}} t^{-\frac{\mu - \nu}{2}} \big\|
    A^{\frac{\nu}{2}} x \big\| \le C k^{\frac{\mu}{2}} t^{-\frac{\mu - \nu}{2}}
    \big\| x \big\|_{\nu}.
  \end{align*}
  Further, the first summand is the error between the exact solution $u$ of
  \eqref{eq3:homeq} and the fully discrete scheme \eqref{eq3:homeq_disc3} at
  time $t_j$. For the case $\mu = \nu = 2$, \cite[Theorem $7.8$]{thomee2006}
  gives the estimate  
  \begin{align*}
    \big\| \big(R(kA_h)^j P_h - E(t_j) \big) x \big\| \le C \big(h^2 + k \big)
    \big\| x \big\|_{2}.   
  \end{align*}
  By the stability of the numerical scheme, that is \eqref{eq3:smoothr} with
  $\rho = 0$, we also have the case $\mu = \nu = 0$. Hence,  
  \begin{align}
    \label{eq3:lem_stab}
    \big\| F_{kh}(t_j) x \big\| \le C \big\| x \big\|,
  \end{align}
  and, as above, the constant $C$ is independent of $h, k \in (0,1]$,
  $t_j > 0$, and $x$. The same interpolation technique, which is used
  in the proof of \cite[Theorem $7.8$]{thomee2006}, gives us the intermediate
  cases for $\mu = \nu$ and $\mu \in [0,2]$, that is
  \begin{align}
    \label{eq3:lem_numu}
    \big\| \big(R(kA_h)^j P_h - E(t_j) \big) x \big\| \le C \big(h^\mu +
    k^{\frac{\mu}{2}} \big) \big\| x \big\|_{\mu}.
  \end{align}
  On the other hand, \cite[Theorem $7.7$]{thomee2006} proves the case
  $\nu = 0$ and $\mu = 2$. Hence, we have
  \begin{align*}
    \big\| \big(R(kA_h)^j P_h - E(t_j) \big) x \big\| \le C \big(h^2 +
    k \big) t_j^{-1} \big\| x \big\|,    
  \end{align*}
  where the constant $C$ is again independent of $h, k \in (0,1]$,
  $t_j > 0$, and $x \in H$. An interpolation between this estimate and
  \eqref{eq3:lem_stab} shows
  \begin{align}
    \label{eq3:lem_nu0}
    \big\| \big(R(kA_h)^j P_h - E(t_j) \big) x \big\| \le C \big(h^\mu +
    k^{\frac{\mu}{2}} \big) t_j^{-\frac{\mu}{2}} \big\| x \big\|    
  \end{align}
  for all $\mu \in [0,2]$. For fixed $\mu \in [0,2]$ the proof of \emph{(i)} is
  completed by an additional interpolation with respect to $\nu \in [0,\mu]$
  between \eqref{eq3:lem_numu} and \eqref{eq3:lem_nu0} and the fact that
  $t_j^{-\frac{\mu}{2}} \le t^{-\frac{\mu}{2}}$.

  The proof of \emph{(ii)} works analogously. The case $\rho = 0$ is true by
  \eqref{eq3:lem_stab} and the case $\rho = 1$ follows by \eqref{eq3:smoothE-1}
  and \eqref{eq3:smoothEkh}, since
  \begin{align*}
    \big\| F_{kh}(t) x \big\| \le \big\| E_{kh}(t) P_h x \big\| + \big\| E(t)
    x \big\| \le C t^{-\frac{1}{2}} \big\| x \big\|_{-1}.
  \end{align*}
  The intermediate cases follow by interpolation.

  For \emph{(iii)} the case $\rho = 0$ is already included in \emph{(i)} with
  $\mu = 2$ and $\nu = 0$. Thus, it remains to show the case $\rho = 1$. For $t
  > 0$ with $t_{j-1} \le t < t_j$ we have
  \begin{align*}
    \big\| F_{kh}(t) x \big\| &\le \big\| \big( R(k A_h)^j - E_h(t_j) \big) P_h
    x \big\| + \big\| \big( E_h(t_j) P_h - E(t_j) \big) x \big\| \\
    &\quad + \big\| \big( E(t_j) - E(t) \big) x \big\| =: T_1 + T_2 +
    T_3.
  \end{align*}
  As in \eqref{eq3:r} we denote by
  $(\lambda_{h,m})_{m=1}^{N_h}$ the positive eigenvalues of $A_h$ with
  cor\-responding orthonormal eigenvectors $(\varphi_{h,m})_{m=1}^{N_h} \subset
  S_h$. For $T_1$ we use the expansion of $P_h x$ in terms of
  $(\varphi_{h,m})_{m=1}^{N_h}$. This yields
  \begin{align*}
    T_1^2 &= \Big\| \sum_{m = 1}^{N_h} \lambda_{h,m}^{\frac{1}{2}} \big( R(k
    \lambda_{h,m} )^j - \ee^{-\lambda_{h,m}t_j} \big) \big( P_h x,
    \lambda_{h,m}^{-\frac{1}{2}} \varphi_{h,m} \big)
    \varphi_{h,m} \Big\|^2 \\
    &= \sum_{m=1}^{N_h} \lambda_{h,m} \big| R(k \lambda_{h,m} )^j -
    \ee^{-k \lambda_{h,m} j} \big|^2 \big( A_{h}^{-\frac{1}{2}} P_h x,
    \varphi_{h,m} \big)^2. 
  \end{align*}
  First, we consider all summands with $k \lambda_{h,m} \le 1$. As in the proof
  of \cite[Theorem $7.1$]{thomee2006}, by applying \eqref{eq3:r1} with $q=1$
  and \eqref{eq3:r2}, we get 
  \begin{align}
    \label{eq3:consr}
    \begin{split}
      \big| R(k \lambda_{h,m} )^j - \ee^{-k \lambda_{h,m} j} \big| &= \Big|
      \big( R(k \lambda_{h,m}) - \ee^{-k \lambda_{h,m}} \big) \sum_{i =
      0}^{j-1} R(k \lambda_{h,m})^{j-1-i} \ee^{-k \lambda_{h,m} i} \Big| \\
      &\le C j \big(k \lambda_{h,m} \big)^{2} \ee^{-c (j-1) k \lambda_{h,m}}.
    \end{split}
  \end{align}
  Therefore, since $t_j = jk$ and $k \lambda_{h,m} \le 1$ it holds true that
  \begin{align*}
    \lambda_{h,m} \big| R(k \lambda_{h,m} )^j - \ee^{-k \lambda_{h,m} j} \big|^2
    &\le C (jk)^{-2} k^2 \lambda_{h,m} \big( j k \lambda_{h,m} )^4 \ee^{-2 c j k
    \lambda_{h,m}} \ee^{2 c k \lambda_{h,m}} \\&\le C t_{j}^{-2} k, 
  \end{align*}
  where we also used that $\sup_{z \ge 0} z^4 \ee^{-2c z} < \infty$. 
  
  For all
  summands with $k \lambda_{h,m} > 1$ we get the estimate
  \begin{align*}
    \lambda_{h,m} \big| R(k \lambda_{h,m} )^j - \ee^{-k \lambda_{h,m} j} \big|^2
    < 2 k^{-1} \big(k \lambda_{h,m} \big)^2 \big( \big| 
    R(k \lambda_{h,m})^j \big|^2 + \big| \ee^{-k \lambda_{h,m} j}\big|^2 \big).
  \end{align*}
  As it is shown in the proof of \cite[Lemma $7.3$]{thomee2006}, we have
  \begin{align}
    \label{eq3:rdecay}
    |R(z)| \le \frac{1}{1+cz},\quad \text{for all } z \ge 1, \text{ with } c >
    0.
  \end{align}
  In fact, for the implicit Euler scheme this is immediately true with $c = 1$,
  but it also holds for all rational functions $R(z)$, which satisfy
  \eqref{eq3:IV}. 
  
  Together with $k \lambda_{h,m} > 1$ this yields
  \begin{align*}
    \big| k \lambda_{h,m}  R(k \lambda_{h,m})^j \big|^2 &\le \Big(
    \frac{k \lambda_{h,m}}{1 + c k \lambda_{h,m}} \Big)^{2}
    \big(1+ c k \lambda_{h,m}\big)^{-2(j-1)}\\
    &\le \frac{1}{c^2} \big( 1 + c \big)^{-2(j-1)} = \frac{1}{c^2}
    \ee^{-2(j-1)\log(1+c)} \le C j^{-2}.
  \end{align*}
  As above we also have
  \begin{align*}
    \big| k \lambda_{h,m} \ee^{-k \lambda_{h,m} j}\big|^2 \le C j^{-2}.     
  \end{align*}
  Therefore, also in the case $k \lambda_{h,m} > 1$, it holds that
  \begin{align*}
    \lambda_{h,m} \big| R(k \lambda_{h,m} )^j - \ee^{-k \lambda_{h,m} j} \big|^2
    \le C t_j^{-2} k.    
  \end{align*}
  Together with Parseval's identity and \eqref{eq3:discnorm} we arrive at
  \begin{align*}
    T_1^2 \le C t_j^{-2} k \sum_{m=1}^{\infty} \big(
    A_h^{-\frac{1}{2}} P_h x, \varphi_{h,m} \big)^2 = C t_j^{-2} k \big\|
    A_h^{-\frac{1}{2}} P_h x \big\|^2 \le C t^{-2} k \big\|
    x \big\|^2_{-1}.
  \end{align*}
  The term $T_2$ is covered by Lemma \ref{lem:Fh1} \emph{(iii)} which gives
  \begin{align*}
    T_2 = \big\| F_h(t_j) x \big\| \le C h t_j^{-1} \big\| x \big\|_{-1} \le C
    h t^{-1} \big\| x \big\|_{-1}.
  \end{align*}
  Finally, for $T_3$ we apply Lemma \ref{lem:1} \emph{(i)} with $\nu = 1$ and
  \emph{(ii)} with $\nu = \frac{1}{2}$ and get
  \begin{align*}
    T_3 = \big\| A E(t) A^{-\frac{1}{2}} \big( E(t_j - t) - I \big)
    A^{-\frac{1}{2}} x \big\| \le C t^{-1} (t_j - t)^{\frac{1}{2}} \big\| x
    \big\|_{-1} \le C t^{-1} k^{\frac{1}{2}} \big\| x
    \big\|_{-1}.
  \end{align*}
  Combining the estimates for $T_1$, $T_2$ and $T_3$ proves
  \emph{(iii)} with $\rho = 1$. As usual, the intermediate cases follow by
  interpolation.
\end{proof}

We also have an analogue of Lemma \ref{lem:Fh2}. A time discrete version of
Lemma \ref{lem:Fkh2} \emph{(ii)}, where the integral is replaced by a sum, is
shown in \cite{yan2005}.

\begin{lemma}
  \label{lem:Fkh2}
  Let $0 \le \rho \le 1$. Under Assumption \ref{as:4} the operator $F_{kh}$
  satisfies the following estimates.

  (i) There exists a constant $C$ such that
  \begin{align*}
    \Big\| \int_{0}^{t} F_{kh}(\sigma) x \ds \Big\| \le C \big(
    h^{2 - \rho} + k^{\frac{2 - \rho}{2}} \big) \big\| x \big\|_{-\rho} 
  \end{align*}
  for all $x \in \dot{H}^{-\rho}$, $t > 0$, and $h, k \in (0,1]$.

  (ii) Under the additional Assumption \ref{as:5} there exists a constant $C$
  such that 
  \begin{align*}
    \Big( \int_{0}^{t} \big\| F_{kh}(\sigma) x \big\|^2 \ds
    \Big)^{\frac{1}{2}} \le C \big( h^{1 + \rho} + k^{\frac{1 + \rho}{2}} \big)
    \big\| x \big\|_{\rho} 
  \end{align*}
  for all $x \in \dot{H}^{-\rho}$, $t > 0$, and $h, k \in (0,1]$.
\end{lemma}

\begin{proof}
  The proof of \emph{(i)} uses a similar technique as the
  proof of Lemma \ref{lem:Fkh1} \emph{(iii)}. First, without loss of
  generality, we can assume that $t = t_n$ for some $n \ge 0$. In fact, if $t_n
  < t < t_{n+1}$ then we have
  \begin{align*}
     \Big\| \int_{0}^{t} F_{kh}( \sigma) x \ds \Big\| \le \Big\| \int_{0}^{t_n}
     F_{kh}( \sigma) x \ds \Big\| + \Big\| \int_{t_n}^{t}
     F_{kh}( \sigma) x \ds \Big\|.
  \end{align*}
  For the second term we get by Lemma \ref{lem:Fkh1} \emph{(iii)}  
  \begin{align*}
    \Big\| \int_{t_n}^{t} F_{kh}( \sigma) x \ds \Big\| &\le \Big\|
    \int_{t_n}^{t} \big( F_{kh}( \sigma) - F_{kh}(t) \big) x \ds \Big\| +
    \Big\| \int_{t_n}^{t} F_{kh}(t) x \ds \Big\| \\
    &= \Big\| \int_{t_n}^{t} \big( E( \sigma) - E(t) \big) x \ds \Big\| +     
    (t - t_n) \big\| F_{kh}(t) x \big\| \\
    &\le \Big\| E(t_n) A^{\frac{\rho}{2}} \int_{t_n}^{t} E( \sigma - t_n )
    A^{-\frac{\rho}{2}} x \ds \Big\| + (t - t_n) \big\| 
    A^{\frac{\rho}{2}} E(t ) A^{-\frac{\rho}{2}} x \big\| \\    
    &\quad + C (t-t_n) \big(
    h^{2 - \rho} + k^{\frac{2 - \rho}{2}} \big) t^{-1} \big\| x
    \big\|_{-\rho}.
  \end{align*}
  We continue by applying Lemma \ref{lem:1} \emph{(iv)} and \emph{(i)}
  with $\nu = 
  \frac{\rho}{2}$ and the fact that $(t - t_n) t^{-1} \le 1$ which yields
  \begin{align*}
    \Big\| \int_{t_n}^{t} F_{kh}( \sigma) x \ds \Big\| & \le C \big(
    (t-t_n)^{1-\frac{\rho}{2}} + (t - t_n) t^{-\frac{\rho}{2}} + h^{2 - \rho} +
    k^{\frac{2 - \rho}{2}} \big) \big\| 
    x \big\|_{-\rho} \\
    &\le C \big(
    h^{2- \rho} +  k^{\frac{2 - \rho}{2}} \big) \big\| x  \big\|_{-\rho} . 
  \end{align*}  
  Next, we have
  \begin{multline*}
    \Big\| \int_{0}^{t_n} F_{kh}( \sigma) x \ds \Big\| \le \Big\|
    \int_{0}^{t_n} \big( E_{kh}( \sigma) - E_h(\sigma) \big) P_h x \ds \Big\|\\
    + \Big\|
    \int_{0}^{t_n} \big( E_{h}( \sigma) P_h - E(\sigma) \big) x \ds \Big\|.    
  \end{multline*}
  For the second term Lemma \ref{lem:Fh2} \emph{(i)} yields the bound
  \begin{align*}
    \Big\|
    \int_{0}^{t_n} \big( E_{h}( \sigma) P_h - E(\sigma) \big)x \ds \Big\| \le C
    h^{2 - \rho} \| x \|_{-\rho}.    
  \end{align*}
  Thus, it is enough to show that 
  \begin{align*}
    \Big\|
    \int_{0}^{t_n} \big( E_{kh}( \sigma) - E_h(\sigma) \big) P_h x \ds \Big\|
    \le C k^{\frac{2- \rho}{2}} \big\| x \big\|_{-\rho},
  \end{align*}
  where the constant $C = C(\rho)$ is independent of $h,k \in (0,1]$, $t>0$,
  and $x \in \dot{H}^{-\rho}$. 
  
  We plug in the definition of $E_{kh}$ and obtain
  \begin{align}
    \label{eq3:est1}
    \begin{split}
      \Big\|
      \int_{0}^{t_n} \big( E_{kh}( \sigma) - E_h(\sigma) \big) P_h x \ds \Big\|
      &\le 
      \Big\| \sum_{j = 1}^n \int_{t_{j-1}}^{t_j} \big( R(k A_h)^j - E_h(t_j)
      \big) P_h x \ds \Big\|\\
      &\quad + \Big\| \sum_{j = 1}^n \int_{t_{j-1}}^{t_j} \big(
      E_h(t_j) - E_h(\sigma) \big) P_h x \ds \Big\|.    
    \end{split}
  \end{align}
  As in \eqref{eq3:r} let $(\lambda_{h,m})_{m=1}^{N_h}$ be the positive
  eigenvalues of $A_h$ with 
  cor\-responding orthonormal eigenvectors $(\varphi_{h,m})_{m=1}^{N_h} \subset
  S_h$. Then, Parseval's identity yields for the first summand 
  \begin{align*}
    &\Big\| \sum_{j = 1}^n \int_{t_{j-1}}^{t_j} \big( R(k A_h)^j - E_h(t_j)
    \big) P_h x \ds \Big\|^2 \\
    &\quad = \sum_{m = 1}^{N_h} \Big| k \sum_{j = 1}^n 
    \big( R(k \lambda_{h,m} )^j - \ee^{-\lambda_{h,m} t_j} \big) \Big|^2 \big(
    P_h x, \varphi_{h,m} \big)^2 \\   
    &\quad \le \sum_{m = 1}^{N_h} \Big( k \sum_{j = 1}^n
    \lambda_{h,m}^{\frac{\rho}{2}}
    \big| R(k \lambda_{h,m} )^j - \ee^{-\lambda_{h,m} t_j} \big| \Big)^2 \big(
    A_h^{-\frac{\rho}{2}} P_h x, \varphi_{h,m} \big)^2.
  \end{align*}
  As in the proof of Lemma \ref{lem:Fkh1} \emph{(iii)} we first study all
  summands with $k \lambda_{h,m} \le 1$. In this case \eqref{eq3:consr} gives
  \begin{align*}
    k \sum_{j = 1}^n
    \lambda_{h,m}^{\frac{\rho}{2}}
    \big| R(k \lambda_{h,m} )^j - \ee^{-\lambda_{h,m} t_j} \big| &\le C k 
    \sum_{j = 1}^n \lambda_{h,m}^{\frac{\rho}{2}} j \big(k \lambda_{h,m}\big)^2
    \ee^{-c(j-1)k \lambda_{h,m}} \\
    &= C \lambda_{h,m}^{\frac{\rho + 4}{2}} \ee^{ck \lambda_{h,m}} k^2 \sum_{j =
    1}^n j k \ee^{-cjk \lambda_{h,m}} \\
    &\le C \lambda_{h,m}^{\frac{\rho + 4}{2}} k \int_{0}^{\infty} (\sigma + k)
    \ee^{-c \lambda_{h,m} \sigma} \ds \\
    & \le C \lambda_{h,m}^{\frac{\rho + 4}{2}} k \Big( \frac{1}{(c
    \lambda_{h,m})^{2}} + \frac{k}{c \lambda_{h,m}} \Big) \le
    C k^{\frac{2 - \rho}{2}}.
  \end{align*}
  For all summands with $k \lambda_{h,m} > 1$ we have the estimates
  \begin{align*}
    & k \sum_{j = 1}^n
    \lambda_{h,m}^{\frac{\rho}{2}}
    \big| R(k \lambda_{h,m} )^j - \ee^{-\lambda_{h,m} t_j} \big|\\ &\quad <
    k^{\frac{2 - \rho}{2}} \sum_{j = 1}^n k \lambda_{h,m} \big( \big| R(k
    \lambda_{h,m}) \big|^j + \ee^{-k\lambda_{h,m} j} \big)  \\ 
    &\quad \le k^{\frac{2 - \rho}{2}} \Big( \frac{k \lambda_{h,m}}{1 + c
    k\lambda_{h,m}} \sum_{j = 1}^n
      \big( 1 + c \big)^{-(j - 1)} +
    k \lambda_{h,m} \ee^{- k \lambda_{h,m}} \sum_{j = 1}^n \ee^{-(j-1)} \Big)
    \le C k^{\frac{2 - \rho}{2}}, 
  \end{align*}
  where we used \eqref{eq3:rdecay} and $\ee^{- k \lambda_{h,m}(j-1)} <
  \ee^{-(j-1)}$ for $k \lambda_{h,m} > 1$. Altogether, this proves
  \begin{multline}
    \label{eq3:term1}
  \Big\| \sum_{j = 1}^n \int_{t_{j-1}}^{t_j} \big( R(k A_h)^j - E_h(t_j)
  \big) P_h x \ds \Big\|^2 \\ \le C k^{2 - \rho} \sum_{j=1}^{N_h} \big(
    A_h^{-\frac{\rho}{2}} P_h x, \varphi_{h,m} \big)^2 = C k^{2 - \rho} \big\|
    A_h^{-\frac{\rho}{2}} P_h x \big\|^2.    
  \end{multline}
  In order to complete the proof of \emph{(i)} it remains to find an estimate
  for the second term in \eqref{eq3:est1}. By applying Parseval's identity in
  the same way as above we get
  \begin{align*}
    &\Big\| \sum_{j = 1}^n \int_{t_{j-1}}^{t_j} \big(
    E_h(t_j) - E_h(\sigma) \big) P_h x \ds \Big\|^2 \\
    &\quad = \sum_{m = 1}^{N_h} \Big|
    \sum_{j=1}^n \int_{t_{j-1}}^{t_j} \big( \ee^{-\lambda_{h,m} t_j} -
    \ee^{- \lambda_{h,m} \sigma} \big) \ds \Big|^2 \big(P_h x, \varphi_{h,m}
    \big)^2 \\
    &\quad = \sum_{m=1}^{N_h} \Big| \lambda_{h,m}^{\frac{\rho}{2}}
    \sum_{j=1}^{n} \ee^{- \lambda_{h,m} t_{j-1}} \int_{0}^{k} \big(
    \ee^{- k \lambda_{h,m}} - \ee^{- \lambda_{h,m}\sigma} \big) \ds
    \Big|^2 \big( A_{h}^{-\frac{\rho}{2}} P_h x, \varphi_{h,m} \big)^2.
  \end{align*}
  Since it holds that
  \begin{align*}
    \int_{0}^{k} \big(\ee^{- k \lambda_{h,m}} - \ee^{- \lambda_{h,m}\sigma}
    \big) 
    \ds &= k \ee^{-k \lambda_{h,m}} - \frac{1}{\lambda_{h,m}} \big( 1 -
    \ee^{- k \lambda_{h,m}} \big)
  \end{align*}
  we have
  \begin{align*}
    &\Big| \lambda_{h,m}^{\frac{\rho}{2}}
    \sum_{j=1}^{n} \ee^{- \lambda_{h,m} t_{j-1}} \int_{0}^{k} \big(
    \ee^{- k \lambda_{h,m}} - \ee^{- \lambda_{h,m}\sigma} \big) \ds
    \Big|^2 \\
    &\quad = \Big| \lambda_{h,m}^{\frac{\rho}{2}} \Big( k \ee^{-k \lambda_{h,m}}
    - \frac{1}{\lambda_{h,m}} \big( 1 - \ee^{- k \lambda_{h,m}} \big) \Big)
    \sum_{j=1}^n \ee^{- k \lambda_{h,m}(j-1)} \Big|^2 \\
    &\quad = \lambda_{h,m}^{\rho - 2} \big| k \lambda_{h,m} \ee^{-k
    \lambda_{h,m}} - \big( 1 - \ee^{- k \lambda_{h,m}} \big)
    \big|^2  \big(1 - \ee^{-k \lambda_{h,m}} \big)^{-2}.
  \end{align*}
  Further, if $k \lambda_{h,m} \le 1$ then it is true that
  \begin{align*}
    \big| k \lambda_{h,m} \ee^{-k \lambda_{h,m}}
    - \big( 1 - \ee^{- k \lambda_{h,m}} \big)
    \big|^2 &= \ee^{- 2 k \lambda_{h,m}} \big| \ee^{k \lambda_{h,m}} - 1 - k
    \lambda_{h,m} \big|^2 \\
    &\le C k^4 \lambda_{h,m}^4.
  \end{align*}
  Thus, in this case we derive the estimate
  \begin{multline*}
    \Big| \lambda_{h,m}^{\frac{\rho}{2}}
    \sum_{j=1}^{n} \ee^{- \lambda_{h,m} t_{j-1}} \int_{0}^{k} \big(
    \ee^{- k \lambda_{h,m}} - \ee^{- \lambda_{h,m}\sigma} \big) \ds
    \Big|^2 \\
    \le C k^{2 - \rho} \big( k \lambda_{h,m} \big)^{\rho}
    \frac{\lambda_{h,m}^2 k^2}{\big(1 - \ee^{-k \lambda_{h,m}} \big)^2} \le C
    k^{2-\rho},
  \end{multline*}
  where we have used that the function $x \mapsto x(1- \ee^{-x})^{-1}$ is
  bounded for all $x \in (0,1]$.
  On the other hand, if $k \lambda_{h,m} > 1$ then we have
  \begin{align*}
    & \Big| \lambda_{h,m}^{\frac{\rho}{2}}
    \sum_{j=1}^{n} \ee^{- \lambda_{h,m} t_{j-1}} \int_{0}^{k} \big(
    \ee^{- k \lambda_{h,m}} - \ee^{- \lambda_{h,m}\sigma} \big) \ds
    \Big|^2 \\
    &\quad \le 2 \lambda_{h,m}^{\rho - 2} \Big( \big| k \lambda_{h,m}
    \ee^{-k \lambda_{h,m}} \big|^2 + \big| 1 - \ee^{-k \lambda_{h,m}} \big|^2
    \Big)
    \big(1 - \ee^{-k\lambda_{h,m}} \big)^{-2} \le C k^{2-\rho}, 
  \end{align*}
  since $\lambda_{h,m}^{\rho - 2} < k^{2 - \rho}$, $\sup_{x \ge 0} x \ee^{-x} <
  \infty$ and $(1- \ee^{-k \lambda_{h,m}})^{-2} \le (1 - \ee^{-1})^{-2}$. 
  Altogether, this yields 
  \begin{align*}
    \Big\| \sum_{j = 1}^n \int_{t_{j-1}}^{t_j} \big(
    E_h(t_j) - E_h(\sigma) \big) P_h x \ds \Big\|^2 &
    \le C k^{2-\rho} \sum_{m=1}^{N_h} \big( A_h^{-\frac{\rho}{2}} P_h x,
    \varphi_{h,m} \big)^2 \\
    &\le C k^{2 - \rho} \big\| A_h^{-\frac{\rho}{2}} P_h x
    \big\|^2,    
  \end{align*}
  which in combination with \eqref{eq3:term1} and \eqref{eq3:discnorm}
  completes the proof of \emph{(i)} for $\rho \in \{ 0, 1 \}$. The intermediate
  cases follow again by interpolation.

  As above we begin the proof of \emph{(ii)} with the remark that without loss
  of generality we can assume that $t = t_n$ for some $n \ge 0$. In a
  similar way as in the proof of \emph{(i)} we have
  \begin{align*}
    \Big( \int_{t_n}^{t} \big\| F_{kh}(\sigma) x \big\|^2 \ds
    \Big)^{\frac{1}{2}} \le \Big( \int_{t_n}^{t} \big\| \big( E(\sigma) - E(t)
    \big) x \big\|^2 \ds \Big)^{\frac{1}{2}} + (t -
    t_n)^{\frac{1}{2}} \big\| F_{kh}(t) x \big\|.
  \end{align*}
  For the second term Lemma \ref{lem:Fkh1} \emph{(i)} with $\mu = 1+ \rho$ and
  $\nu = \rho$ together with $(t-t_n)t^{-1} \le 1$ gives the desired estimate.
  The first summand is estimated by Lemma \ref{lem:1} \emph{(ii)} which gives
  \begin{align*}
    \Big( \int_{t_n}^{t} \big\| \big( E(\sigma) - E(t)
    \big) x \big\|^2 \ds \Big)^{\frac{1}{2}} &= \Big( \int_{t_n}^{t} \big\|
    E(\sigma) A^{- \frac{\rho}{2}} \big( I - E(t - \sigma) \big)
    A^{\frac{\rho}{2}} x \big\|^2 \ds \Big)^{\frac{1}{2}} \\
    &\le \Big( \int_{t_n}^{t} (t- \sigma)^{\rho} \ds
    \Big)^{\frac{1}{2}} \big\| x \big\|_{\rho} \le C k^{\frac{1 + \rho}{2}}
    \big\| x \big\|_{\rho}. 
  \end{align*}
  Further, we have
  \begin{multline*}
    \Big( \int_{0}^{t_n} \big\| F_{kh}(\sigma) x \big\|^2 \ds
    \Big)^{\frac{1}{2}} \le \Big( \int_{0}^{t_n} \big\| \big( E_{kh}(\sigma) -
    E_h(\sigma) \big) P_h x \big\|^2 \ds \Big)^{\frac{1}{2}} \\
    +  \Big( \int_{0}^{t_n} \big\| \big( E_{h}(\sigma) P_h - 
    E(\sigma) \big) x \big\|^2 \ds \Big)^{\frac{1}{2}}
  \end{multline*}
  and Lemma \ref{lem:Fh2} \emph{(ii)} yields an estimate for the second summand
  of the form
  \begin{align*}
    \Big( \int_{0}^{t_n} \big\| \big( E_{h}(\sigma) P_h - 
    E(\sigma) \big) x \big\|^2 \ds \Big)^{\frac{1}{2}} \le C h^{1 + \rho}
    \big\| x \big\|_\rho.    
  \end{align*}
  Thus it remains to show
  \begin{align*}
    \Big( \int_{0}^{t_n} \big\| \big( E_{kh}(\sigma) -
    E_h(\sigma) \big) P_h x \big\|^2 \ds \Big)^{\frac{1}{2}} \le C
    k^{\frac{1+ \rho}{2}} \big\| x \big\|_{\rho}.
  \end{align*}
  As above, we prove this estimate for $\rho \in \{0,1\}$. The intermediate
  cases follow again by interpolation. By the definition of $E_{kh}$ we obtain
  the analogue of \eqref{eq3:est1}, namely
  \begin{align}
    \label{eq3:est2}
    \begin{split}
      &\Big( \int_{0}^{t_n} \big\| \big( E_{kh}(\sigma) - E_h(\sigma) \big) P_h
      x \big\|^2 \ds \Big)^{\frac{1}{2}}\\
      &\quad \le \Big( \sum_{j=1}^n
      \int_{t_{j-1}}^{t_j} \big\| \big( R(k A_h)^j - E_h(t_j) \big) P_h x
      \big\|^2 \ds \Big)^{\frac{1}{2}} \\
      &\qquad + \Big( \sum_{j=1}^n \int_{t_{j-1}}^{t_j} \big\| \big( E_h(t_j) -
      E_h(\sigma) \big) P_h x \big\|^2 \ds \Big)^{\frac{1}{2}}.
    \end{split}
  \end{align}
  For the square of the first summand we again apply Parseval's identity with
  respect to the orthonormal eigenbasis $(\varphi_{h,m})_{m = 1}^{N_h} \subset
  S_h$ of $A_h$ with corresponding eigenvalues
  $(\lambda_{h,m})_{m=1}^{N_h}$ and get
  \begin{align*}
      & \sum_{j=1}^n
      \int_{t_{j-1}}^{t_j} \big\| \big( R(k A_h)^j - E_h(t_j) \big) P_h x
      \big\|^2 \ds\\
      &\quad = \sum_{j=1}^n k \sum_{m = 1}^{N_h}
      \lambda_{h,m}^{-\rho} \big| R(k
      \lambda_{h,m})^j - \ee^{-k \lambda_{h,m}j} \big|^2 \big(
      A_{h,m}^{\frac{\rho}{2}} P_h x,
      \varphi_{h,m} \big)^2. 
  \end{align*}
  For all summands with $k \lambda_{h,m} \le 1$ we apply \eqref{eq3:consr}.
  This yields
  \begin{align}
    \label{eq3:term5}
    \begin{split}
      & \sum_{j=1}^n
      \int_{t_{j-1}}^{t_j} \big\| \big( R(k A_h)^j - E_h(t_j) \big) P_h x
    \big\|^2 \ds\\
    & \quad \le C \sum_{m = 1}^{N_h} k \lambda_{h,m}^{-\rho} \sum_{j= 1}^{n}
    j^2 \big(k \lambda_{h,m}\big)^4 \ee^{- 2c(j-1) k \lambda_{h,m}} \big(
    A_h^{\frac{\rho}{2}} P_h x, \varphi_{h,m} \big)^2\\
    & \quad \le C \sum_{m = 1}^{N_h} k^2 \lambda_{h,m}^{4 - \rho} \ee^{2ck
    \lambda_{h,m}} \int_{0}^{\infty} ( \sigma + k)^2 \ee^{-2c \lambda_{h,m}
    \sigma} \ds \big( A_h^{\frac{\rho}{2}} P_h x, \varphi_{h,m} \big)^2\\
    & \quad \le C \sum_{m = 1}^{N_h} k^2 \lambda_{h,m}^{4 - \rho} \Big(
    \frac{2}{(2c \lambda_{h,m})^3} + \frac{2k}{(2c \lambda_{h,m})^2} +
    \frac{k^2}{2 c \lambda_{h,m}} \Big) \big( A_h^{\frac{\rho}{2}} P_h x,
    \varphi_{h,m} \big)^2\\
    & \quad \le C k^{1+ \rho} \big\| A_h^{\frac{\rho}{2}} P_h x \big\|^2.
  \end{split}
  \end{align}
  For the remaining summands with $k \lambda_{h,m} > 1$ we use
  \eqref{eq3:rdecay} and get the estimate
  \begin{align*}
    \sum_{j=1}^n \big| R(k \lambda_{h,m})^j - \ee^{- k\lambda_{h,m} j} \big|^2 &
    \le 2 \sum_{j=1}^n \big( ( 1 + c )^{-2j} + \ee^{-2j} \big) \le C, 
  \end{align*}
  where the bound $C$ is independent of $n$. Hence, also in this case we have
  \begin{align}
    \label{eq3:term6}
    \begin{split}
    & \sum_{j=1}^n
    \int_{t_{j-1}}^{t_j} \big\| \big( R(k A_h)^j - E_h(t_j) \big) P_h x
    \big\|^2 \ds\\
    & \le C \sum_{m = 1}^{N_h} k \lambda_{h,m}^{-\rho} \big(
    A_h^{\frac{\rho}{2}} P_h x, \varphi_{h,m} \big)^2 < C k^{1 + \rho} \big\|
    A_h^{\frac{\rho}{2}} P_h x \big\|^2.   
  \end{split}
  \end{align}
  Next, we prove a similar result for the square of the second summand in
  \eqref{eq3:est2}. As in the proof of part \emph{(i)} an application of
  Parseval's identity yields
  \begin{align*}
    &\sum_{j = 1}^{n} \int_{t_{j-1}}^{t_j} \big\| \big(E_h(t_j) - E_h(\sigma)
    \big) P_h x \big\|^2 \ds\\
    &\quad = \sum_{j = 1}^{n} \sum_{m = 1}^{N_h} \lambda_{h,m}^{-\rho}
    \int_{t_{j-1}}^{t_j} 
     \big| \ee^{-k\lambda_{h,m}j} - \ee^{-\lambda_{h,m}\sigma}
    \big|^2 \ds \big(A_{h}^{\frac{\rho}{2}} P_h x, \varphi_{h,m} \big)^2.
  \end{align*}
  By using the fact
  \begin{align*}
    \int_{t_{j-1}}^{t_j} \big| \ee^{- k \lambda_{h,m} j} - \ee^{-
    \lambda_{h,m} \sigma} \big|^2 \ds \le \ee^{-2 k \lambda_{h,m}(j-1)} k \big(
    1 - \ee^{-k \lambda_{h,m}} \big)^2
  \end{align*}
  we obtain 
  \begin{align*}
    &\sum_{j = 1}^{n} \int_{t_{j-1}}^{t_j} \big\| \big(E_h(t_j) - E_h(\sigma)
    \big) P_h x \big\|^2 \ds\\
    &\quad \le \sum_{m = 1}^{N_h} \lambda_{h,m}^{-\rho} k \big(1 -
    \ee^{-k \lambda_{h,m}} \big)^2 \big(A_{h}^{\frac{\rho}{2}} P_h x,
    \varphi_{h,m} \big)^2\sum_{j=1}^{n} 
    \ee^{-2 k \lambda_{h,m}(j-1)} \\
    &\quad \le \sum_{m = 1}^{N_h} \lambda_{h,m}^{-\rho} k \big(1 -
    \ee^{-k \lambda_{h,m}} \big)^2 \big(A_{h}^{\frac{\rho}{2}} P_h x,
    \varphi_{h,m} \big)^2 \big( 1 - \ee^{-2k \lambda_{h,m}} \big)^{-1}.
  \end{align*}
  Since $1 - \ee^{-2k \lambda_{h,m}} = (1 + \ee^{-k \lambda_{h,m}}) ( 1 -
  \ee^{-k \lambda_{h,m}})$  we conclude
  \begin{align*}
    &\sum_{j = 1}^{n} \int_{t_{j-1}}^{t_j} \big\| \big(E_h(t_j) - E_h(\sigma)
    \big) P_h x \big\|^2 \ds\\
    &\quad \le \sum_{m = 1}^{N_h} \lambda_{h,m}^{-\rho} k \big(1 -
    \ee^{-k \lambda_{h,m}} \big)
    \big(A_{h}^{\frac{\rho}{2}} P_h x, \varphi_{h,m} \big)^2 \big( 1 + \ee^{-k
    \lambda_{h,m}} \big)^{-1}\\ 
    &\quad \le \frac{1}{2} k \sum_{m = 1}^{N_h}
    \lambda_{h,m}^{-\rho} \big(1 - \ee^{-k \lambda_{h,m}} \big)
    \big(A_{h}^{\frac{\rho}{2}} P_h x, \varphi_{h,m} \big)^2 
  \end{align*}
  If $\rho = 0$ this simplifies to $\frac{1}{2} k \| P_h x \|^2$ since $1-
  \ee^{-k \lambda_{h,m}} \le 1$. If $\rho = 1$ then we use $1 - \ee^{-k
  \lambda_{h,m}} \le k \lambda_{h,m}$ and the right hand side is bounded by
  $\frac{1}{2} k^2 \| A_h^{\frac{1}{2}} P_h x \|^2$. Altogether this proves
  \begin{align}
    \label{eq3:term7}
    \sum_{j = 1}^{n} \int_{t_{j-1}}^{t_j} \big\| \big(E_h(t_j) - E_h(\sigma)
    \big) P_h x \big\|^2 \ds \le C k^{1 + \rho} \big\| A_h^{\frac{1}{2}} P_h x
    \big\|^2.    
  \end{align}
  Finally, a combination of \eqref{eq3:est2} with \eqref{eq3:term5},
  \eqref{eq3:term6} and \eqref{eq3:term7} gives
  \begin{align}
     \Big( \int_{0}^{t_n} \big\| \big( E_{kh}(\sigma) - E_h(\sigma) \big) P_h
     x \big\|^2 \ds \Big)^{\frac{1}{2}} \le C k^{\frac{1+ \rho}{2}} \big\|
     A_h^{\frac{\rho}{2}} P_h x \big\|    
    \label{eq3:term8}
  \end{align}
  which completes the proof of the case $\rho = 0$. For the case $\rho = 1$ we
  additionally use \eqref{eq3:normAh} and Assumption \ref{as:5} which yields
  \begin{align*}
    \big\| A_h^{\frac{1}{2}} P_h x \big\| = \big\| P_h x \big\|_{1} \le C \| x
    \|_{1}
  \end{align*}
  for all $x \in \dot{H}^1$. This completes the proof of the Lemma.
\end{proof}

\sect{Proof of Theorem \ref{th:semi}}
\label{sec:semi}
The aim of this section is to prove the strong convergence of the spatially 
semidiscrete approximation \eqref{eq4:mildsemi} to the mild solution
\eqref{eq1:mild} of \eqref{eq1:SPDE}. 

The following two lemmas are useful for the proof of Theorem \ref{th:semi}.
The first lemma is a special version of \cite[Lemma $7.2$]{daprato1992} and
is needed to estimate the stochastic integrals. The second is a
generalized version of Gronwall's lemma. A proof of this version can be
found in \cite{elliott1992} or in \cite[Lemma 7.1.1]{henry1981}.  

\begin{lemma}
  \label{lem:stochint}
  For any $p \in [2, \infty)$, $t \in [0,T]$ and for any $L_2^0$-valued
  predictable process $\Phi(\sigma)$, $\sigma \in [0,t]$, we have
  \begin{align*}
    \E \Big[ \Big\| \int_{0}^{t} \Phi(\sigma) \dWs \Big\|^p \Big] \le C(p) \E
    \Big[ \Big( \int_{0}^{t} \| \Phi(\sigma) \|^2_{L^0_2} \ds
    \Big)^{\frac{p}{2}} \Big].
  \end{align*}
  Here the constant $C(p)$ can be chosen to be
  \begin{align*}
    C(p) = \Big( \frac{p}{2}(p-1)\Big)^{\frac{p}{2}} \Big(
    \frac{p}{p-1}\Big)^{p(\frac{p}{2}-1)}.
  \end{align*}
\end{lemma}

\begin{lemma}
  \label{lem:Gronwall}
  Let the function $\varphi \colon [0,T] \to \R$ be nonnegative and
  continuous. If 
  \begin{align}
    \varphi(t) = C_1 + C_2 \int_{0}^{t} (t-\sigma)^{-1+\beta}
    \varphi(\sigma) \ds    
    \label{eq4:1}
  \end{align}
  for some constants $C_1,C_2 \ge 0$ and $\beta > 0$ and for all $t \in
  (0,T]$ , then there exists
  a constant $C = C(C_2,T,\beta)$ such that
  \begin{align*}
    \varphi(t) \le C C_1, \text{ for all } t\in (0,T].
  \end{align*}
\end{lemma}

After these preparations we are ready to prove the first main theorem.

\begin{proof}[Proof of Theorem \ref{th:semi}]
  For $t \in (0,T]$ we have by \eqref{eq1:mild} and \eqref{eq4:mildsemi} 
  \begin{align}
    \label{eq4:main}
    \begin{split}
      &\big\| X_h(t) - X(t) \big\|_{\LOH} \le \big\| F_h(t) X_0 \big\|_{\LOH} \\
      &\quad + \Big\| \int_{0}^{t} E_h(t-\sigma) P_h f(X_h(\sigma)) \ds -
      \int_{0}^{t} E(t-\sigma) f(X(\sigma)) \ds \Big\|_{\LOH} \\
      &\quad + \Big\| \int_{0}^{t} E_h(t-\sigma) P_h g(X_h(\sigma)) \dWs -
      \int_{0}^{t} E(t-\sigma) g(X(\sigma)) \dWs \Big\|_{\LOH},
    \end{split}
  \end{align}
  where $F_h(t) = E_h(t) P_h - E(t)$. The first term is estimated by
  Lemma \ref{lem:Fh1} \emph{(i)} with $\mu = \nu = 1+r$, which yields
  \begin{align}
    \big\| F_h(t) X_0 \big\|_{\LOH} \le C h^{1+r} \big\|
    A^{\frac{1+r}{2}} X_0 \big\|_{\LOH}.    
    \label{eq4:I}
  \end{align}
  The second term in \eqref{eq4:main} is dominated by three additional terms as
  follows
  \begin{align*}
    & \Big\| \int_{0}^{t} E_h(t-\sigma) P_h f(X_h(\sigma)) \ds -
    \int_{0}^{t} E(t-\sigma) f(X(\sigma)) \ds \Big\|_{\LOH} \\
    &\quad \le \Big\| \int_{0}^{t} E_h(t-\sigma) P_h \big( f(X_h(\sigma)) -
    f(X(\sigma)) \big) \ds \Big\|_{\LOH} \\
    &\qquad + \Big\| \int_{0}^{t} \big( E_h(t-\sigma) P_h - E(t-\sigma) \big) 
    \big( f(X(\sigma)) - f(X(t)) \big) \ds \Big\|_{\LOH} \\
    &\qquad + \Big\| \int_{0}^{t} \big( E_h(t-\sigma) P_h - E(t-\sigma) \big) 
    f(X(t)) \ds \Big\|_{\LOH}\\
    &\quad =: I_1 + I_2 + I_3.
  \end{align*}
  We estimate each term separately. First note that, by interpolation
  between \eqref{eq3:smoothEh} and \eqref{eq3:smoothEh2}, we have $\|
  E_h(t) P_h x \| \le C t^{-\frac{1-r}{2}} \| x \|_{-1 + r}$. Together with
  Assumption \ref{as:2} this yields 
  \begin{align}
    \label{eq4:I1}
    \begin{split}
    I_1 &\le \int_{0}^{t} \big\|  E_h(t-\sigma) P_h \big( f(X_h(\sigma)) -
    f(X(\sigma)) \big) \big\|_{\LOH} \ds \\
    &\le C \int_{0}^{t} (t-\sigma)^{-\frac{1 - r}{2}} \big\| X_h(\sigma) -
    X(\sigma) \big\|_{\LOH} \ds.
  \end{split}
  \end{align} 
  The term $I_2$ is estimated by applying Lemma \ref{lem:Fh1} \emph{(iii)} with
  $\rho = 1 - r$, Assumption \ref{as:2} and \eqref{eq1:hoelder} with $s = 0$.
  Then we get 
  \begin{align}
    \begin{split}
      I_2 &\le \int_{0}^{t} \big\| F_h(t-\sigma) \big( f(X(\sigma)) -
      f(X(t)) 
      \big) \big\|_{\LOH} \ds \\
      &\le C h^{1+r} \int_{0}^{t} (t-\sigma)^{-1} \big\|X(\sigma) - X(t)
      \big\|_{\LOH} \ds \\
      &\le C h^{1+r} \int_{0}^{t} (t-\sigma)^{-1 + \frac{1}{2}}
      \ds \le C h^{1+r}. 
    \end{split}
    \label{eq4:I2}
  \end{align}
  Finally, the estimate for $I_3$ is a straightforward application of Lemma
  \ref{lem:Fh2} \emph{(i)} $\rho = 1 - r$. A further application of Assumption
  \ref{as:2} and \eqref{eq1:reg} gives 
  \begin{align}
    \begin{split}
      I_3 \le C h^{1+r} \big\| f(X(t)) \big\|_{-1+r} \le C h^{1+r} \Big(1 +
      \sup_{\sigma \in [0,T]} \big\| X(\sigma) \big\|_{\LOH} \Big) \le C
      h^{1+r}.
    \end{split}
    \label{eq4:I3}
  \end{align}
  The right hand side of this estimate is finite in view of
  \eqref{eq1:reg}. A combination of the estimates \eqref{eq4:I1},
  \eqref{eq4:I2}, and \eqref{eq4:I3} yields
  \begin{align}
    \begin{split}
    & \Big\| \int_{0}^{t} E_h(t-\sigma) P_h f(X_h(\sigma)) \ds -
    \int_{0}^{t} E(t-\sigma) f(X(\sigma)) \ds \Big\|_{\LOH} \\
    & \quad \le C h^{1+r} + C \int_{0}^{t} (t-\sigma)^{\frac{r-1}{2}} \big\|
    X_h(\sigma) - X(\sigma) \big\|_{\LOH} \ds.
    \end{split}
    \label{eq4:II}
  \end{align}
  Next, we estimate the norm of the stochastic integral in
  \eqref{eq4:main}. First, we apply Lemma \ref{lem:stochint} and get
  \begin{align*}
    &\Big\| \int_{0}^{t} E_h(t-\sigma) P_h g(X_h(\sigma)) \dWs -
    \int_{0}^{t} E(t-\sigma) g(X(\sigma)) \dWs \Big\|_{\LOH}\\
    &\quad \le C \Big( \E \Big[ \Big( \int_{0}^{t} \big\| E_h(t-\sigma) P_h
    g(X_h(\sigma)) - E(t-\sigma) g(X(\sigma)) \big\|^2_{L_2^0} \ds
    \Big)^{\frac{p}{2}} \Big] \Big)^{\frac{1}{p}}.    
  \end{align*}
  The right hand side is a norm. Hence, the triangle
  inequality gives 
  \begin{align*}
    &\Big( \E \Big[ \Big( \int_{0}^{t} \big\| E_h(t-\sigma) P_h
    g(X_h(\sigma)) - E(t-\sigma) g(X(\sigma)) \big\|^2_{L_2^0} \ds
    \Big)^{\frac{p}{2}} \Big] \Big)^{\frac{1}{p}}\\
    &\quad \le \Big\| \Big( \int_{0}^{t} \big\| E_h(t-\sigma) P_h \big(
    g(X_h(\sigma)) - g(X(\sigma)) \big) \big\|^2_{L_2^0} \ds
    \Big)^{\frac{1}{2}} \Big\|_{\LOR} \\
    &\qquad + \Big\| \Big( \int_{0}^{t} \big\| F_h(t-\sigma) \big(
    g(X(\sigma)) - g(X(t)) \big) \big\|^2_{L_2^0} \ds
    \Big)^{\frac{1}{2}} \Big\|_{\LOR} \\
    &\qquad + \Big\| \Big( \int_{0}^{t} \big\| F_h(t-\sigma) g(X(t))
    \big\|^2_{L_2^0} \ds \Big)^{\frac{1}{2}} \Big\|_{\LOR}\\
    &\quad =: I_4 + I_5 + I_6.
  \end{align*}
  In a similar way as for $I_1$, we find an estimate for $I_4$ by using the
  stability of the operator $E_h(t) P_h$, that is, \eqref{eq3:smoothEh} with
  $\rho = 0$. Together with Assumption \ref{as:1} we get
  \begin{align}
    \label{eq4:I4}
    \begin{split}
      I_4 &\le C \Big\| \Big( \int_{0}^{t} \big\| X_h(\sigma) - X(\sigma)
      \big\|^2 \ds \Big)^{\frac{1}{2}} \Big\|_{\LOR} \\
      &= C \left( \Big\| \int_{0}^{t} \big\| X_h(\sigma) - X(\sigma)
      \big\|^2 \ds \Big\|_{L^{p/2}(\Omega;\mathbb{R})} \right)^{\frac{1}{2}}\\
      &\le C \Big( \int_{0}^{t} \big\| X_h(\sigma) - X(\sigma)
      \big\|^2_{\LOH} \ds
      \Big)^{\frac{1}{2}}. 
    \end{split}   
  \end{align}
  For the estimate of term $I_5$ we apply Lemma \ref{lem:Fh1} \emph{(i)} with
  $\mu = 1+r$ 
  and $\nu = 0$, which gives $\| F_h(t)\| \le C h^{1+r}
  t^{-\frac{1+r}{2}}$. Additionally, we use the fact that $\| L M \|_{L^0_2}
  \le \| L \| \| M \|_{L_2^0}$ for any bounded linear operator $L\colon H \to
  H$ and Hilbert-Schmidt operator $M \in L^0_2$. The estimate is completed
  by making use of Assumption \ref{as:1} and the  
  H\"older-continuity \eqref{eq1:hoelder} with $s=0$. Altogether, we derive 
  \begin{align}
    \label{eq4:I5}
    \begin{split}
      I_5 &\le C h^{1+r} \Big\| \Big( \int_{0}^{t} (t-\sigma)^{-1-r} \big\|
      X(\sigma) - X(t) \big\|^2 \ds \Big)^{\frac{1}{2}} \Big\|_{\LOR}\\
      &\le C h^{1+r} \Big( \int_{0}^{t} (t-\sigma)^{-1-r} \big\| X(\sigma) -
      X(t) \big\|^2_{\LOH} \ds \Big)^{\frac{1}{2}}\\
      &\le C h^{1+r} \Big( \int_{0}^{t} (t-\sigma)^{-r} \ds
      \Big)^{\frac{1}{2}} \le C h^{1+r}.
    \end{split}
  \end{align}
  Note that in the last step the generic constant $C$ depends on $r \in [0,1)$
  and blows up as $r \to 1$.

  Finally, for $I_6$, let $(\varphi_m)_{m\ge1}$ denote an
  arbitrary orthonormal basis of the Hilbert space $U_0$. Then, by using Lemma
  \ref{lem:Fh2} \emph{(ii)} with $\rho = r$, Assumption \ref{as:1} and
  \eqref{eq1:reg}, we get
  \begin{align}
    \label{eq4:I6}
    \begin{split}
    I_6 &= \Big\| \Big( \sum_{m=1}^{\infty} \int_{0}^{t} \big\| F_h(t-\sigma)
    g(X(t)) \varphi_m \big\|^2 \ds \Big)^{\frac{1}{2}} \Big\|_{\LOR} \\
    &\le C h^{1+r} \Big\| \Big( \sum_{m=1}^{\infty} \big\|
    A^{\frac{r}{2}} g(X(t)) \varphi_m \big\|^2 \Big)^{\frac{1}{2}}
    \Big\|_{\LOR} \\
    &= C h^{1+r} \Big\|\, \big\| g(X(t)) \big\|_{L_{2,r}^0} \, \Big\|_{\LOR}\\
    &\le C h^{1+r} \Big( 1 +  \sup_{\sigma \in [0,T]} \big\| A^{\frac{r}{2}}
    X(\sigma) \big\|_{\LOH} \Big) \le C h^{1+r}.
  \end{split}
  \end{align}
  In total, we have by \eqref{eq4:I4}, \eqref{eq4:I5}, and \eqref{eq4:I6} that
  \begin{align}
    \label{eq4:III}
    \begin{split}
      &\Big\| \int_{0}^{t} E_h(t-\sigma) P_h g(X_h(\sigma)) \dWs -
    \int_{0}^{t} E(t-\sigma) g(X(\sigma)) \dWs \Big\|_{\LOH}\\
    &\quad\le C h^{1+r} + C \Big( \int_{0}^{t} \big\| X_h(\sigma) - X(\sigma)
      \big\|^2_{\LOH} \ds
      \Big)^{\frac{1}{2}}.
    \end{split}
  \end{align}
  Coming back to \eqref{eq4:main}, 
  by \eqref{eq4:I}, \eqref{eq4:II}, and \eqref{eq4:III} we conclude that
  \begin{align*}
    \| X_h(t) - X(t) \|^2_{\LOH} &\le C h^{2(1+r)} + C
    \int_{0}^{t} \big\| X_h(\sigma) - X(\sigma) \|^2_{\LOH} \ds\\
    & \quad + C \Big( \int_{0}^{t} (t-\sigma)^{-\frac{1-r}{2}} \big\|
    X_h(\sigma) - X(\sigma) \big\|_{\LOH} \ds \Big)^2.     
  \end{align*}
  Finally, we note that
  \begin{align*}
    &\int_{0}^{t} \big\| X_h(\sigma) - X(\sigma) \|^2_{\LOH} \ds\\
    &\quad = \int_{0}^{t}
    (t-\sigma)^{\frac{1 - r}{2}} (t-\sigma)^{-\frac{1-r}{2}} \big\| X_h(\sigma)
    - X(\sigma) \|^2_{\LOH} \ds \\
    &\quad \le T^{\frac{1-r}{2}} \int_{0}^{t}
    (t-\sigma)^{-\frac{1-r}{2}} \big\| X_h(\sigma) -
    X(\sigma) \|^2_{\LOH} \ds,    
  \end{align*}
  and by H\"older's inequality
  \begin{align*}
    &\int_{0}^{t} (t-\sigma)^{-\frac{1-r}{2}} \big\|
    X_h(\sigma) - X(\sigma) \big\|_{\LOH} \ds \\
    &\quad = \int_{0}^{t}
    (t-\sigma)^{-\frac{1-r}{4}} (t-\sigma)^{-\frac{1-r}{4}} \big\|
    X_h(\sigma) - X(\sigma) \big\|_{\LOH} \ds \\
    &\quad \le \Big( \frac{2}{r+1} T^{\frac{r+1}{2}}
    \Big)^{\frac{1}{2}} \Big( \int_{0}^{t} (t-\sigma)^{-\frac{1-r}{2}} \big\|
    X_h(\sigma) - X(\sigma) \|^2_{\LOH} \ds \Big)^{\frac{1}{2}}. 
  \end{align*}
  Hence, for $\varphi(t) = \big\|X_h(t) - X(t) \big\|^2_{\LOH}$ we have shown
  that
  \begin{align*}
    \varphi(t) \le C h^{2(1+r)} + C \int_{0}^{t}
    (t-\sigma)^{\frac{r-1}{2}} \varphi(\sigma) \ds
  \end{align*}
  and Lemma \ref{lem:Gronwall} completes the proof.
\end{proof}

\sect{Proof of Theorem \ref{th5:full}}
\label{sec:full}

This section is devoted to the proof of Theorem \ref{th5:full}. As in Section
\ref{sec:semi} the proof relies on a discrete 
version of Gronwall's Lemma. Here we use a variant from \cite[Lemma
7.1]{elliott1992}.

\begin{lemma}
  \label{lem:discGronwall}
  Let $C_1, C_2 \ge 0$ and let $(\varphi_j)_{j =
  1, \ldots, N_k}$ be a nonnegative sequence. 
  If for $\beta \in (0,1]$ we have
  \begin{align}
    \varphi_j \le C_1 + C_2 k \sum_{i = 1}^{j-1} t_{j-i}^{-1 + \beta} \varphi_i
    \quad \text{for all } j = 1,\ldots,N_k,
    \label{eqA:2}
  \end{align}
  then there exists a constant $C = C(C_2, T, \beta)$ such that
  \begin{align*}
    \varphi_j \le C C_1 \quad \text{ for all } j=1, \ldots, N_k.
  \end{align*}
  In particular, the constant $C$ does not depend on $k$.
\end{lemma}

\begin{proof}[Proof of Theorem \ref{th5:full}]
  In terms of the rational function $R(k A_h)$, which was introduced in
  \eqref{eq3:r}, we derive the following discrete variation of constants
  formula for $X_h^j$
  \begin{align}
  \label{eq5:mild}
  X_h^j &= R(k A_h)^j P_h X_0 - k \sum_{i = 0}^{j-1} R(k A_h)^{j-i}
  P_h f(X_h^{i}) + \sum_{i = 0}^{j-1} R(k A_h)^{j-i} P_h g(X_h^i) \Delta
  W^{i+1}.
\end{align}

  By applying \eqref{eq1:mild} and \eqref{eq5:mild} we get for the error
  \begin{align}
    \label{eq5:1}
    \begin{split}
    &\| X_h^j - X(t_j) \|_{\LOH} \le \big\| R(k A_h)^j P_h X_0 - E(t_j) X_0
    \big\|_{\LOH} \\
    &\quad + \Big\| k \sum_{i = 0}^{j-1} R(k A_h)^{j -i} P_h f(X_h^i)
    - \int_{0}^{t_j} E(t_j - \sigma) f(X(\sigma)) \ds \Big\|_{\LOH} \\
    &\quad + \Big\| \sum_{i = 0}^{j-1} R(k A_h)^{j-i} P_h g(X_h^i) \Delta
    W^{i+1} - \int_{0}^{t_j} E(t_j - \sigma) g(X(\sigma)) \dWs
    \Big\|_{\LOH}.
  \end{split}
  \end{align}
  The first summand is the error of the fully discrete approximation scheme
  \eqref{eq3:homeq_disc} for the homogeneous equation \eqref{eq3:homeq} but with
  the initial value being a random variable. By Assumption \ref{as:3} we have
  that $X_0(\omega) \in \dot{H}^{1+r}$ for $P$-almost all $\omega \in \Omega$.
  Thus, the error estimate from \cite[Theorem 7.8]{thomee2006} yields
  \begin{align*}
    \big\| R(k A_h)^j P_h X_0 - E(t_j) X_0 \big\|_{\LOH} \le 
    C \big(h^{1+r} + k^{\frac{1+r}{2}} \big) \big\| A^{\frac{1+r}{2}} X_0
    \big\|_{\LOH}. 
  \end{align*}
  For the other two summands we introduce an auxiliary process which is given by
  \begin{align}
    \label{eq5:defXkh}
    \begin{split}
    X_{kh}(t) &:= X_h^{j-1}, \quad \text{if } t \in (t_{j-1}, t_j], \quad j = 1,
    \ldots, N_k,\\
    X_{kh}(0) &:= P_h X_0.
  \end{split}
  \end{align}
  By this definition $X_{kh}$ is an adapted and left-continuous process and,
  therefore, predictable \cite[p.~27]{roeckner2007}. Recalling the definition
  \eqref{eq3:defEkh} of 
  the family of operators $E_{kh}(t)$, $t \ge 0$, we obtain
  \begin{align}
    \begin{split}
    k \sum_{i=0}^{j-1} R(k A_h)^{j -i} P_h f(X_h^i) &= \sum_{i=0}^{j-1}
    \int_{t_i}^{t_{i+1}} E_{kh}(t_j - \sigma) P_h f( X_{kh}(\sigma)) \ds \\
    &= \int_{0}^{t_j} E_{kh}(t_j - \sigma) P_h f( X_{kh}(\sigma)) \ds
  \end{split}
  \label{eq5:int1}
  \end{align}
  and, analogously,
  \begin{align}
    \label{eq5:int2}
    \sum_{i=0}^{j-1} R(k A_h)^{j-i} P_h g(X_h^i) \Delta W^{i+1} =
    \int_0^{t_j} E_{kh}(t_j - \sigma) P_h g(X_{kh}(\sigma)) \dWs.    
  \end{align}
  By applying \eqref{eq5:int1} we have the following estimate of the second
  summand in \eqref{eq5:1}
  \begin{align*}
    &\Big\| k \sum_{i = 0}^{j-1} R(k A_h)^{j -i} P_h f(X_h^i)
    - \int_{0}^{t_j} E(t_j - \sigma) f(X(\sigma)) \ds \Big\|_{\LOH} \\
    &\quad = \Big\| \int_{0}^{t_j} \big( E_{kh}(t_j - \sigma) P_h f(
    X_{kh}(\sigma)) - E(t_j - \sigma) f(X(\sigma)) \big) \ds \Big\|_{\LOH} \\
    &\quad \le \Big\| \int_{0}^{t_j} E_{kh}(t_j - \sigma) P_h \big(
    f(X_{kh}(\sigma)) - f(X(\sigma)) \big) \ds \Big\|_{\LOH} \\
    &\quad \quad + \Big\| \int_{0}^{t_j} \big( E_{kh}(t_j - \sigma) P_h - E(t_j
    - \sigma) \big) \big( f(X(\sigma)) - f(X(t_j)) \big) \ds \Big\|_{\LOH}  \\
    &\quad \quad + \Big\| \int_{0}^{t_j} \big( E_{kh}(t_j - \sigma) P_h - E(t_j
    - \sigma) \big) f(X(t_j)) \ds \Big\|_{\LOH}  \\
    &\quad =: J_1 + J_2 + J_3.
  \end{align*}
  Note, that the terms $J_1$, $J_2$, and $J_3$ are of the same structure as the
  terms $I_1$, $I_2$, and $I_3$ in the proof of Theorem \ref{th:semi}. Since
  with Lemmas \ref{lem:Fkh1} and \ref{lem:Fkh2} we have the time discrete
  analogues of Lemmas \ref{lem:Fh1} and \ref{lem:Fh2} at our disposal the proof
  follows the same path.  
  
  For the term $J_1$ we first note that an interpolation between
  \eqref{eq3:smoothr} and \eqref{eq3:smoothEkh} yields $\big\| E_{kh}(t) P_h
  x \big\| \le C t_i^{-\frac{1-r}{2}} \big\| x \big\|_{-1+r}$ for all $t \in
  [t_{i-1},t_i)$ and all $x \in
  \dot{H}^{-1 + r}$. Together with Assumption \ref{as:2} this gives
  \begin{align*}
    \begin{split}
      J_1 &\le \int_{0}^{t_j} \big\| E_{kh}(t_j - \sigma) P_h \big(
      f(X_{kh}(\sigma)) - f(X(\sigma)) \big) \big\|_{\LOH} \ds \\
      &\le C \sum_{i = 0}^{j-1} \int_{t_i}^{t_{i+1}}
      t_{j-i}^{-\frac{1-r}{2}} \big\| X_h^{i} - X(\sigma) \big\|_{\LOH}
      \ds \\
      &\le C k \sum_{i = 0}^{j-1} t_{j-i}^{-\frac{1-r}{2}} \big\| X_h^{i} -
      X(t_i) \big\|_{\LOH} \\
      &\quad + C \sum_{i = 0}^{j-1} t_{j-i}^{-\frac{1-r}{2}}
      \int_{t_i}^{t_{i+1}} \big\| X(t_i) - X(\sigma)  \big\|_{\LOH} \ds. \\
    \end{split}
  \end{align*}
  By the H\"older continuity \eqref{eq1:hoelder} we get
  \begin{multline*}
    \sum_{i = 0}^{j-1} t_{j-i}^{-\frac{1-r}{2}}
      \int_{t_i}^{t_{i+1}} \big\| X(t_i) - X(\sigma)  \big\|_{\LOH} \ds\\
      \le C \sum_{i = 0}^{j-1} t_{j-i}^{-\frac{1-r}{2}} 
      \int_{t_i}^{t_{i+1}} (\sigma - t_i)^{\frac{1}{2}} \ds 
      = \frac{2}{3} C k^{\frac{3}{2}} \sum_{i = 0}^{j-1}
      t_{j-i}^{-\frac{1-r}{2}} \\
      \le C k^{\frac{1}{2}} \int_{0}^{t_j} \sigma^{- \frac{1-r}{2}} \ds  \le C
      k^{\frac{1}{2}}.
  \end{multline*}  
  Altogether, $J_1$ is estimated by
  \begin{align}
    \label{eq5:J1}
    \begin{split}
      J_1 &\le C k^{\frac{1}{2}} + C k \sum_{i = 0}^{j-1}
      t_{j-i}^{-\frac{1-r}{2}} \big\| X_h^{i} - 
      X(t_i) \big\|_{\LOH}.
    \end{split}
  \end{align}
  For the estimate of $J_2$ we first note that $F_{kh}(t) = E_{kh}(t)P_h- E(t)$
  and, hence, we can apply Lemma \ref{lem:Fkh1} \emph{(iii)} with $\rho = 1 -
  r$. In the same way as in \eqref{eq4:I2} we obtain
  \begin{align}
    \label{eq5:J2}
    J_2 &\le C \big( h^{1+r} + k^{\frac{1+r}{2}} \big).
  \end{align}
  Likewise, but this time by an application of Lemma \ref{lem:Fkh2} \emph{(i)}
  with $\rho = 1 - r$, we proceed with the term $J_3$ in same way as in
  \eqref{eq4:I3}. By also using \eqref{eq1:reg} this yields
  \begin{align}
    \label{eq5:J3}
    J_3 \le C \big(h^{1+r} + k^{\frac{1+r}{2}} \big) \Big( 1 + \sup_{\sigma \in
    [0,T]} \big\| X(\sigma) \big\|_{\LOH} \Big) \le C \big( h^{1+r} +
    k^{\frac{1+r}{2}} \big). 
  \end{align}
  It remains to estimate the third summand in \eqref{eq5:1} which contains the
  stochastic integrals. With \eqref{eq5:int2} and Lemma \ref{lem:stochint} we
  get 
  \begin{align*}
    & \Big\| \sum_{i = 0}^{j-1} R(k A_h)^{j-i} P_h g(X_h^i) \Delta
    W^{i+1} - \int_{0}^{t_j} E(t_j - \sigma) g(X(\sigma)) \dWs
    \Big\|_{\LOH} \\
    &\quad = \Big\| \int_0^{t_j} E_{kh}(t_j - \sigma) P_h g(X_{kh}(\sigma))
    \dWs - \int_{0}^{t_j} E(t_j - \sigma) g(X(\sigma)) \dWs
    \Big\|_{\LOH}\\
    &\quad \le C \Big( \E \Big[ \Big( \int_{0}^{t_j} \big\| E_{kh}(t_j -
    \sigma) P_h g(X_{kh}(\sigma)) - E(t_j - \sigma) g(X(\sigma))
    \big\|_{L_2^0}^2 \ds \Big)^{\frac{p}{2}} \Big] \Big)^{\frac{1}{p}}.
  \end{align*}
  In the last step Lemma \ref{lem:stochint} is applicable
  since by our definitions \eqref{eq3:defEkh} and \eqref{eq5:defXkh} the
  process $[0, t_j] \ni \sigma \mapsto E_{kh}(t_j - \sigma)P_h
  g(X_{kh}(\sigma)) \in L^0_2$ is adapted and left-continuous and, therefore,
  predictable.

  Next, we use the triangle inequality and obtain
  \begin{align*}
    & \Big( \E \Big[ \Big( \int_{0}^{t_j} \big\| E_{kh}(t_j -
    \sigma) P_h g(X_{kh}(\sigma)) - E(t_j - \sigma) g(X(\sigma))
    \big\|_{L_2^0}^2 \ds \Big)^{\frac{p}{2}} \Big] \Big)^{\frac{1}{p}}\\
    &\quad \le \Big\| \Big( \int_{0}^{t_j} \big\| E_{kh}(t_j - \sigma) P_h \big(
    g(X_{kh}(\sigma)) - g(X(\sigma)) \big) \big\|^2_{L^0_2} \ds
    \Big)^{\frac{1}{2}} \Big\|_{\LOR} \\
    &\quad \quad + \Big\| \Big( \int_{0}^{t_j} \big\| F_{kh}(t_j - \sigma)
    \big( g(X(\sigma)) - g(X(t_j)) \big) \big\|^2_{L^0_2} \ds
    \Big)^{\frac{1}{2}} \Big\|_{\LOR} \\
    &\quad \quad + \Big\| \Big( \int_{0}^{t_j} \big\| F_{kh}(t_j - \sigma)
    g(X(t_j)) \big\|^2_{L^0_2} \ds \Big)^{\frac{1}{2}} \Big\|_{\LOR}\\
    &\quad =: J_4 + J_5 + J_6.
  \end{align*}
  For the estimate of $J_4$ we use the facts that $\| E_{kh}(t) P_h x \| \le C
  \| x \|$ for all $x \in H$ as well as $\| L M \|_{L^0_2} \le \| L \| \| M
  \|_{L_2^0}$ for all linear bounded operators $L \colon H \to H$ and $M \in
  L_2^0$. Together with Assumption \ref{as:1} and the same technique as in
  \eqref{eq4:I4} we get
  \begin{align*}
    J_4 &\le C \Big\| \Big( \int_{0}^{t_j} \big\| X_{kh}(\sigma) - X(\sigma)
    \big\|^2 \ds \Big)^{\frac{1}{2}} \Big\|_{\LOR} \\
    &\le C \Big( \int_{0}^{t_j} \big\|X_{kh}(\sigma) - X(\sigma)
    \big\|^2_{\LOH} \ds \Big)^{\frac{1}{2}}\\
    &= C \Big( k \sum_{i=0}^{j-1}  \big\| X_h^i -
    X(t_i)\big\|^2_{\LOH} \Big)^{\frac{1}{2}} \\
    &\quad + C \Big( \sum_{i=0}^{j-1} 
    \int_{t_i}^{t_{i+1}} \big\|X(t_i) - X(\sigma) \big\|^2_{\LOH} \ds
    \Big)^{\frac{1}{2}}.
  \end{align*}
  By the H\"older continuity \eqref{eq1:hoelder} it holds that
  \begin{align*}
    \sum_{i=0}^{j-1} 
    \int_{t_i}^{t_{i+1}} \big\|X(t_i) - X(\sigma) \big\|^2_{\LOH} \ds
     & \le C \sum_{i=0}^{j-1} 
    \int_{t_i}^{t_{i+1}} (\sigma - t_i) \ds
     \le C k.    
  \end{align*}
  Hence, we have
  \begin{align}
    \label{eq5:J4}
    J_4 \le C k^{\frac{1}{2}} + C \Big( k \sum_{i = 0}^{j-1} \big\| X_h^i -
    X(t_i) \big\|^2_{\LOH} \Big)^{\frac{1}{2}}.
  \end{align}
  In the way same as in \eqref{eq4:I5}, we apply Lemma \ref{lem:Fkh1}
  \emph{(i)} with $\mu = 1 + r$ and $\nu = 0$ and we derive the estimate
  \begin{align}
    \label{eq5:J5}
    J_5 \le C \big( h^{1+r} + k^{\frac{1+r}{2}} \big),
  \end{align}
  where we also used Assumption \ref{as:1} and \eqref{eq1:hoelder}. By
  \eqref{eq1:reg} and the same
  arguments which gave \eqref{eq4:I6} we get
  \begin{align}
    \label{eq5:J6}
    J_6 \le C \big( h^{1+r} + k^{\frac{1+r}{2}} \big).
  \end{align}
  To summarize our estimate of \eqref{eq5:1}, by \eqref{eq5:J1} to
  \eqref{eq5:J6} we have shown that
  \begin{align*}
    \big\| X_h^j - X(t_j) \big\|_{\LOH}^2 &\le C \big( h^{1+r} +
    k^{\frac{1}{2}} + k^{\frac{1+r}{2}} \big)^2 \\&\quad + C \Big( k \sum_{i =
    0}^{j-1} t_{j-i}^{-\frac{1-r}{2}} \big\| X_h^{i} - X(t_i) \big\|_{\LOH}
    \Big)^2 \\
    &\quad + C  k \sum_{i = 0}^{j-1} \big\| X_h^i -
    X(t_i) \big\|^2_{\LOH}.
  \end{align*}
  Further, we have
  \begin{align*}
    k \sum_{i = 0}^{j-1} \big\| X_h^i -
    X(t_i) \big\|^2_{\LOH} \le T^{\frac{1-r}{2}} k \sum_{i = 0}^{j-1}
    t_{j-i}^{-\frac{1-r}{2}} \big\| X_h^i - X(t_i) \big\|^2_{\LOH},    
  \end{align*}
  and by H\"older's inequality 
  \begin{align*}
    &k \sum_{i = 0}^{j-1} t_{j-i}^{-\frac{1-r}{2}} \big\| X_h^{i} - X(t_i) 
    \big\|_{\LOH} \\&\quad \le \Big( k \sum_{i = 0}^{j-1}
    t_{j-i}^{-\frac{1-r}{2}} 
    \Big)^{\frac{1}{2}} \Big( k \sum_{i = 0}^{j-1} t_{j-i}^{-\frac{1-r}{2}}
    \big\| X_h^{i} - X(t_i) \big\|_{\LOH}^2 \Big)^{\frac{1}{2}} \\
    &\quad \le \Big( \frac{2}{r + 1} T^{\frac{r+1}{2}} \Big)^{\frac{1}{2}}
    \Big( k \sum_{i = 0}^{j-1} t_{j-i}^{-\frac{1-r}{2}} 
    \big\| X_h^{i} - X(t_i) \big\|_{\LOH}^2 \Big)^{\frac{1}{2}}.
  \end{align*}
  Hence, by setting $\varphi_j = \big\| X_h^j - X(t_j) \big\|^2_{\LOH}$ we have
  proven that
  \begin{align*}
    \varphi_j \le C \big( h^{1+r} +
    k^{\frac{1}{2}} + k^{\frac{1+r}{2}} \big)^2 + C k \sum_{i=0}^{j-1}
    t_{j-i}^{-\frac{1-r}{2}} \varphi_{i}.
  \end{align*}
  An application of Lemma \ref{lem:discGronwall} completes the proof of the
  theorem.
\end{proof}

\sect{Additive noise}
\label{sec:addnoise}

In this section we focus on stochastic partial differential equations with
additive noise, that is, the $L_2^0$-valued function $g$ does not depend on
$X$. Thus, the SPDE \eqref{eq1:SPDE} has the form
\begin{align}
  \label{eq6:SPDE}
  \begin{split}
  \mathrm{d}X(t) + \big[A X(t) + f(X(t)) \big] \dt &= g \dWt, \quad \text{ for }
  0 \le t \le T,\\
  X(0) &= X_0.
\end{split}
\end{align}

Numerical schemes for the approximation of SPDEs with additive noise 
have been
extensively studied. For example, for schemes which involve the finite element
method we refer to \cite{hausenblas2003,kovacs2010,yan2004} and 
the references therein.

For additive noise Assumption \ref{as:1} simplifies to 
\begin{assumption}
  \label{as:6}
  There exists $r \in [0,1]$ such that the Hilbert-Schmidt operator $g \in
  L_2^0$ satisfies 
  \begin{align*}
    \| g \|_{L_{2,r}^0}  < \infty.
  \end{align*}  
\end{assumption}
Recall that in the case, where $U = H$ and $g\colon H \to H$ is the identity,
Assumption \ref{as:6} reads as follows
\begin{align*}
  \| g \|_{L_{2,r}^0} = \| A^{\frac{r}{2}} g \|_{L_2^0} = \sum_{m=1}^{\infty}
  \| A^{\frac{r}{2}} Q^{\frac{1}{2}} \varphi_m \|^2 < \infty,
\end{align*}
where $(\varphi_m)_{m \ge 1}$ denotes an arbitrary orthonormal basis of $H$.
In particular, if $r=0$ we have $\| g \|_{L_{2}^0} = \mathrm{Tr}(Q) < \infty$
which is a common assumption on the covariance operator $Q$ (see
\cite{daprato1992,roeckner2007}).

Under Assumptions \ref{as:2}, \ref{as:3} and \ref{as:6} with $r\in [0,1]$, $p
\in [2, \infty)$ there exists a mild solution $X\colon [0,T] \times \Omega \to
\dot{H}^{1+r}$ to \eqref{eq6:SPDE} (see \cite{daprato1992} and \cite[Corollary
5.2]{kl2010a}). 

In particular, we stress the fact that the parameter value $r =
1$ is included for SPDEs with additive noise. As the next corollary shows the
same is true for the error estimates of the numerical approximations
\eqref{eq4:mildsemi} and \eqref{eq5:scheme}.

\begin{corollary}
  Let Assumptions \ref{as:2}, \ref{as:3} and
  \ref{as:6} hold for some $r \in [0,1]$, $p \in [2, \infty)$. Let $X$ denote
  the mild solution to \eqref{eq6:SPDE}.

  (i) Under Assumption \ref{as:4} there exists a constant $C$, independent of
  $h  \in (0,1]$, such that
  \begin{align*}
    \big\| X_h(t) - X(t) \big\|_{\LOH} \le C h^{1+r}, \quad \text{ for all } t
    \in (0,T],
  \end{align*}
  where $X_h$ is the corresponding spatially semidiscrete approximation
  \eqref{eq4:mildsemi}.

  (ii) Under Assumptions \ref{as:4} and \ref{as:5} there exists a constant $C$,
  independent of $k,h \in (0,1]$, such that
  \begin{align*}
    \big\| X_h^j - X(t_j) \big\|_{\LOH} \le C ( h^{1+r} +
    k^{\frac{1}{2}} \big), \quad \text{for all } j = 1, \ldots, N_k,
  \end{align*}
  where $X_h^j$ is the corresponding spatio-temporal discrete
  approximation \eqref{eq5:scheme}.  
\end{corollary}

\begin{proof}
  For \emph{(i)} and \emph{(ii)} the assertion follows directly from Theorems
  \ref{th:semi} 
  and \ref{th5:full} for all parameter values $r \in [0,1)$. 

  A close inspection of the proof of Theorem \ref{th:semi} shows that the
  condition $r<1$ is
  only required in the estimate \eqref{eq4:I5}. But in the case of additive
  noise the term $I_5$ is equal to $0$. Hence, the convergence result
  holds true by the same arguments for $r = 1$.

  The same is true for \emph{(ii)} where the condition $r<1$ only shows up in
  the estimate of \eqref{eq5:J5}.
\end{proof}

We remark that a similar convergence result to \emph{(i)} can be found in
\cite[Prop.\ 2.3]{kovacs2010}. But there the authors had to incorporate a
singularity of the form $\max(0, \log(\frac{t}{h^2}))$ which is now removed. 

\subsection*{Acknowledgment} 
Most of the research in this article has been carried out during a
stay at Chalmers University of Technology. I would
like to thank my host Prof.~S. Larsson and his working group for the generous
hospitality and the inspiring environment they created for me. My thanks also
go to my supervisor Prof.~W.-J.~Beyn and the 
German Academic Research Service
(DAAD) for making the stay possible.

\def\cprime{$'$} \def\polhk#1{\setbox0=\hbox{#1}{\ooalign{\hidewidth
  \lower1.5ex\hbox{`}\hidewidth\crcr\unhbox0}}}

\end{document}